\documentclass[11pt,reqno,a4paper]{amsart}

\usepackage{array}
\usepackage{enumitem}
\usepackage{mathtools}
\newcommand\scalemath[2]{\scalebox{#1}{\mbox{\ensuremath{\displaystyle #2}}}}
\usepackage{pgfplots}
\usepackage{soul}

\usepackage{amsmath,amssymb,amsfonts}
\usepackage{tikz}
\usepackage{hyperref}
\usepackage{appendix}
\usepackage{scalerel}
\usepackage{graphicx}
\usepackage{subcaption}
\usepackage{amsthm}
\usepackage{multirow}
\usepackage{pdflscape}
\usepackage{afterpage}
\usepackage{capt-of} 
\usepackage{lipsum}
\usepackage{booktabs}
\usepackage{siunitx}
\usepackage{esvect}

\tikzset{decorated arrows/.style={
    postaction={
        decorate,
        decoration={
            markings,
            mark=between positions 0 and 1 step 15mm with {\arrow[black]{stealth};}
            }
        },
    }
}
 
\tikzset{decorated arrows2/.style={
    postaction={
        decorate,
        decoration={
            markings,
            mark=at position 15mm with {\arrow[black]{stealth};}
            }
        },
    }
}

\usepgfplotslibrary{colormaps}
\usetikzlibrary{arrows}

\makeatletter
\def\set@curr@file#1{%
  \begingroup
    \escapechar\m@ne
    \xdef\@curr@file{\expandafter\string\csname #1\endcsname}%
  \endgroup
}
\def\quote@name#1{"\quote@@name#1\@gobble""}
\def\quote@@name#1"{#1\quote@@name}
\def\unquote@name#1{\quote@@name#1\@gobble"}
\makeatother
\usepackage{graphics}

\theoremstyle{plain}
\newtheorem{thm}{Theorem}
\newtheorem{lem}[thm]{Lemma}
\newtheorem{prop}[thm]{Proposition}

\newtheorem{maintheorem}{Theorem}

 \theoremstyle{definition}
\newtheorem{defn}{Definition}[section]

\theoremstyle{remark}
\newtheorem{rem}{Remark}

\theoremstyle{plain}

\newcommand{\RR}{\mathbb{R}}

\newcommand{\dpt}{\displaystyle}

\newcommand{\RN}[1]{%
  \textup{\uppercase\expandafter{\romannumeral#1}}%
}
      
\author[J. P. S. M. de Carvalho and A. A. Rodrigues]{Jo\~ao P.S. Maur\'icio de Carvalho$^{*1}$ and Alexandre A. Rodrigues$^{2}$ \\
\\
$^*$\MakeLowercase{up200902671@up.pt} \\ $^2$\MakeLowercase{alexandre.rodrigues@fc.up.pt} \\
\\
$^{1,2}$Faculty of Sciences, University of Porto, \\ Rua do Campo Alegre s/n, Porto 4169-007, Portugal \\ \\
$^2$Centre for Mathematics, University of Porto, \\ Rua do Campo Alegre s/n, Porto 4169-007, Portugal \\
}

\begin{document}

\subjclass[2010]{37D45, 37G10, 37G15, 03C25}
\keywords{SIR model, Seasonality, Basic reproduction number, Backward bifurcation, Strange attractors, Observable chaos} 
\thanks{JPSMC was supported by Project MAGIC POCI-01-0145-FEDER-032485, funded by FEDER via
COMPETE 2020 - POCI and by FCT/MCTES via PIDDAC. AR was partially supported by CMUP (UIBD/MAT/00144/2020), which is funded by Funda\c{c}\~ao para a Ci\^encia e a Tecnologia (FCT) with national  and European structural funds through the programs FEDER, under the partnership agreement PT2020. AR also benefits from the grant CEECIND/01075/2020 of the Stimulus of Scientific Employment -- 3rd Edition (Individual Support)  awarded by FCT.  \\ $^*$Corresponding author.}

\title[Strange attractors in a modified SIR model]
{Strange attractors in a dynamical system \\ inspired by a seasonally forced SIR model}

\date{\today}

\begin{abstract}

We analyze a multiparameter periodically-forced dynamical system inspired in the SIR endemic model. We show that the condition on the \emph{basic reproduction number} $\mathcal{R}_0 < 1$ is not sufficient to guarantee the elimination of \emph{Infectious} individuals due to a \emph{backward bifurcation}. Using the theory of rank-one attractors, for an open subset in the space of parameters where $\mathcal{R}_0<1$, the flow exhibits \emph{persistent strange attractors}. These sets are not confined to a tubular neighbourhood in the phase space, are numerically observable and shadow the ghost of a two-dimensional invariant torus.
Although numerical experiments have already suggested that periodically-forced biological models may exhibit observable chaos, a rigorous proof was not given before. Our results agree well with the empirical belief that intense seasonality induces chaos.

This work provides a preliminary investigation of the interplay between seasonality, deterministic dynamics and the prevalence of strange attractors in a nonlinear forced system inspired by biology.
\end{abstract}

\maketitle

\section{Introduction}\label{sec:intro}

The emergence of mathematical models associated to epidemiology has made an important contribution to fight against a wide range of diseases, such as AIDS, tuberculosis, hepatitis and most recently CoViD-19 \cite{Carvalho2020, Bonyah2020, Rajagopal2020, Cobey2020, Britton2010}. Simple models have been generalised in various ways  in order to decide about preventive measures to contain the disease.  

The SIR model \cite{Kermack1932,Dietz1976} is one of the simplest compartmental models, and many models come from this basic form. It consists of three compartments: susceptible (S), infectious (I) and recovered (R) individuals, and is reasonably predictive for infectious diseases that are transmitted from human to human, and where recovery confers  resistance, such as measles, mumps and rubella \cite{ParkBolker2020, Keeling2001}. In general, SIR models have a global attractor in a homogeneous environment \cite{Keeling2001}.

Although for some specific diseases the impact of seasonality is minor and can be safely neglected in modeling them, in other cases, for example for the childhood diseases and for influenza, this impact is extremely important and must be explicitly modeled. Indeed, the current state of research indicates empirical evidence of the ubiquity of seasonal forces in epidemic models, including factors that influence disease dynamics  over time, such as school hours, climate change, human phenomena, environmental changes, political decisions, among others \cite{Buonomo2018}.  For example, seasonal flu is a striking example where seasonal forces play a crucial role since there are periods of the year  when the incidence of this flu has a high impact \cite{Moghadami2017}.

In mathematical models that include \emph{seasonal forcing}, transmission rates can be modulated through periodic functions \cite{Barrientos2017, Duarte2019, Rashidinia2018} -- they are more realistic in this type of cases. These non-autonomous differential equations add further levels of complexity to classical models.

\subsection{State of the art on periodically-perturbed models}
In 2001, Keeling {\it et al.}~\cite{Keeling2001} have  {analyzed}  a seasonally forced SIR  model (whose attention is focused on the dynamics of measles, whooping cough and rubella) and  concluded that the dynamics of diseases with more impact on children (who have been subjected to seasonality) is more complex, contrary to what had been expected until then. 
 Bilal {\it et al.}~\cite{Bilal2016} studied the dynamics of various types of models applied to epidemiology where the rate of disease transmission was modulated through a periodic function and concluded that the emergence of strange non-chaotic attractors predicted the growth of epidemics.
 
In 2017, Barrientos \emph{et al.}~\cite{Barrientos2017} aimed to understand to what extent the consequences of seasonality had an impact on epidemic models and have shown, analytically, the existence of topological horseshoes (chaos) in the sense of \cite{Medio2009}  under  the existence of seasonality in the transmission rate of the disease, low birth and mortality rates, and high rates of recovery and transmission.  These   horseshoes are hyperbolic, have zero Lebesgue measure and are invisible in terms of numerics.

The \emph{basic reproduction number}, denoted by $\mathcal{R}_0$, may be seen as a threshold parameter, intended
to quantify the spread of disease by estimating the average number of secondary infections, in a completely susceptible population,
giving an indication of the invasion strength of an epidemic \cite{Li2011}. It measures the
number of secondary cases generated by an infectious case once an epidemic is ongoing.

Nowadays, $\mathcal{R}_0$ has been widely used as a measure of disease strength to estimate the effectiveness of control measures and to form the backbone
of disease-management policy. 
Statistically, if $\mathcal{R}_0 < 1$, then the spread of the disease slows down and is eliminated,  whereas if $\mathcal{R}_0 > 1$, then the disease persists \cite{Jones2007}.
However, in dynamical models, this information about $\mathcal{R}_0$ may fail: diseases can persist with $\mathcal{R}_0 < 1$  {\cite{Li2011}}. 

\subsection{Novelty}

The contribution of this paper to the literature is twofold. First, we exhibit a multiparameter dynamical system inspired by the SIR endemic model with $\mathcal{R}_0<1$  {for which the  {\emph{Infectious}} component does not vanish.} Second, we prove that, under a seasonal periodic forcing {$\Phi$} with high frequency  {$\omega$}, \emph{strange attractors} appear persistently in its flow.  

The rigorous proof of the strange character of an invariant set is a great challenge and the proof of the \emph{abundance} (with respect to the Lebesgue measure) of such attractors is a very involved task. 
 Although the model under analysis may not be biologically realistic, the persistence of chaotic dynamics is  {relevant} because it means that the phenomenon is numerically \emph{observable} (in the phase space) and \emph{persistent} (it occurs with positive probability in the space of parameters).

 \subsection{Structure}
 We analyze a periodically-forced {dynamical system inspired by the SIR endemic model to investigate the influence of seasonality on the disease dynamics.} In Section \ref{THEMODEL} we describe and motivate the structure of our model, compute the \emph{basic reproduction number} ($\mathcal{R}_0$) and state the main results of the study.
We show in Section \ref{compact} that the flow is positively flow-invariant when restricted to a compact set.
In Sections \ref{DFE_lab} and \ref{END_lab} we study the equilibria  and we present the proof of the first main result. Also, we briefly analyze the sensitivity of the \emph{basic reproduction number} with respect to the parameters of the dynamical system.
We prove {in} Section \ref{SA_lab} our second main result. 
Finally, in Section \ref{discussion} we discuss the results and relate with others in the literature.  

\section{Setting and main results}
\label{THEMODEL}
In this section, we introduce the model under consideration and we state the main results, as well as the structure of the article. 
\subsection{Model}
 We are going to divide the individuals of a given population into three classes of individuals \cite{Dietz1976,LiTeng2017}:

\smallskip

\begin{itemize}
\item \emph{Susceptible (S)}: number of individuals that are currently not infected, but can contract the disease;  
\medskip
\item  \emph{Infectious (I)}: number of individuals who are currently infected and can actively transmit  the disease to a susceptible individual, until their recovery;
\medskip
\item \emph{Recovered (R)}: number of individuals who currently can neither be infected, nor can infect susceptible individuals. This comprises individuals who have  definitive immunity  because they have recovered from a recent infection.
\end{itemize}

\smallskip

\noindent The model under consideration assumes that the susceptible individuals have never been in contact with the disease. However, they can become infected and belong to the class of infectious individuals who support the spread of the disease. When they recover they are immune to the disease. Inspired in \cite{Dietz1976,LiTeng2017, ZhangChen1999,Perez2019}, the nonlinear system of ordinary differential equations (ODE) in the variables $S$, $I$ and $R$ (depending on the time $t$), is given by the following one-parameter family:

\begin{equation}
\label{modelo}
\begin{array}{lcl}
\dot{X} = \mathcal{F}_\gamma(X) \quad \Leftrightarrow \quad
\begin{cases}
&\dot{S} =  S(A-S) - \beta_\gamma(t)IS \\
\\
&\dot{I} =  \beta_\gamma(t)IS -  {(\mu + d) I} - \dfrac{r I}{a + I} \\
\\	
&\dot{R} =  \dfrac{r I}{a + I} - \mu R,
\end{cases}
\end{array}
\end{equation}

\smallskip

\noindent where 

\smallskip

$$
\begin{array}{lcl}
X(t) &=& \left( S(t), I(t), R(t) \right), \\
\\
\dot{X} &=& (\dot{S}, \dot{I}, \dot{R}) \,\,\, = \,\,\, \dpt \left(\frac{\mathrm{d}S}{\mathrm{d}t},\frac{\mathrm{d}I}{\mathrm{d}t},\frac{\mathrm{d}R}{\mathrm{d}t}\right), \\
\\
\beta_\gamma(t) &=& \beta_0 \left(1+\gamma \Phi(\omega t)\right).
\end{array}
$$

\bigskip

\noindent The vector field associated to \eqref{modelo} will be called by $\mathcal{F}_\gamma$ and the associated flow is $\varphi_\gamma \left(t, (S_0, I_0, R_0)\right)$, $t \in \mathbb{R}_0^+$, $(S_0, I_0, R_0) \in (\mathbb{R}_0^+)^3$. Figure \ref{boxes} illustrates the interaction between the classes of susceptible, infectious and recovered individuals in model \eqref{modelo}.
 
\usetikzlibrary{arrows,positioning}
\begin{center}
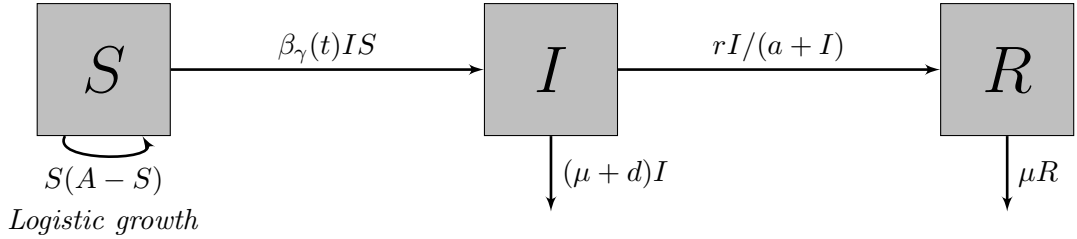
\begin{figure}[ht!]
\begin{tikzpicture}
[
auto,
>=latex',
every node/.append style={align=center},
int/.style={draw, minimum size=1.75cm}
]

    \node [fill=lightgray,int] (S)             {\Huge $S$};
    \node [fill=lightgray,int, right=5cm] (I) {\Huge $I$};
    \node [fill=lightgray,int, right=11cm] (R) {\Huge $R$};
    
    \node [below=of I] (i) {} ;
    \node [below=of R] (r) {};
    
    \coordinate[right=of I] (out);
    \path[->, auto=false,line width=0.35mm] (S) edge node {$\beta_\gamma(t) I S$ \\[.6em]} (I)
                          (I) edge node {$rI/(a+I)$       \\[.6em] } (R) 
                          (S) edge  [out=-120, in=-60] node[below] {$S(A-S)$ \\ [0.2em] \emph{Logistic growth}} (S);

    \path[->, auto=false,line width=0.35mm] (I) edge [] node[right]{${(\mu + d) I}$} (i) ;
    
    \path[->, auto=false,line width=0.35mm] (R) edge [] node[right]{$\mu R$} (r) ;

\end{tikzpicture}
\caption{\small Schematic diagram of model \eqref{modelo}. Boxes represent compartments, and arrows indicate the flow between boxes.}
\label{boxes}
\end{figure}
\end{center}

\subsection{Interpretation of the constants}
The parameters of \eqref{modelo} may be interpreted as follows: \\

\begin{description}
\item[$A$] carrying capacity of susceptible people when $\beta_0=0$ \emph{i.e.} in the absence of disease; \\
\item[$\gamma$] amplitude of the seasonal variation that oscillates between $ \beta_0\left (1+\gamma \min_{t\in [0,T]} \Phi(t)\right)>0$ in the low season, and $\beta_0\left (1+\gamma \max_{t\in [0,T]} \Phi(t)\right)$ in the high season, for some $T>0$;  \\
\item[ $\Phi(\omega t)$]  effects of periodic seasonality on $\beta_0$ over the time with frequency $\omega>0$;  \\
\item[ $\mu$] natural death rate of infected and recovered individuals;\\
 \item[ $d$] death rate of infected individuals due to the disease;\\
\item[ $r$] cure rate;\\
\item[ $a$] measures the effects of a delay in the response treatment (proportional to the saturation of health services; see Remark \ref{Rem1} later);\\
\item[ $\beta_0$] transmission rate of the disease when $\gamma=0$ \emph{i.e.} in the absence of seasonality. The parameter $\gamma$ ``measures the deformation'' of the transmission rate due to the seasonality.\\
\end{description}

\subsection{Motivation}
System \eqref{modelo} has been inspired in the classical SIR model \cite{Kermack1932} by the reasons we proceed to explain:  \\
\begin{itemize}
\item In the \emph{Susceptible} population, we have considered the {\emph{logistic growth}} $S(A-S)$ instead of an exponential growth as in \cite[Equation (1.4)]{LiTeng2017} and \cite[Equation (1)]{ZhangChen1999}; \\

\item   The disease transmission rate $\beta_\gamma$ is given by a non-autonomous periodic function able to capture seasonal variations \cite{Dietz1976,Keeling2001} instead of a constant map; \\ 

\item  The transition from  {\emph{Infectious}} to \emph{Recovered} is the homographic function $\frac{r I}{a + I}$ since the medical conditions are limited and do not grow linearly with the number of  {{\it Infectious}} (see Remark \ref{Rem1} and \cite[Subsection 2.2]{ZhangLiu2008}). In contrast to the findings of \cite{Perez2019,ZhangLiu2008} our choice stresses the effects of a delay in the response treatment.

\end{itemize}

\begin{rem}
\label{Rem1}
A different transition rate from  {\emph{Infectious}} to \emph{Recovered}  has been studied in \cite{Perez2019} \emph{via} the saturated Holling type II treatment rate
$$
T_{(\lambda, \varepsilon)}(I)=\dfrac{\lambda I}{1 + \varepsilon I}, 
$$
 where   $\lambda>0$ is the maximal treatment rate for each individual per  unit of time and $\varepsilon>0$  {is the constant that measures the saturation effect caused by the infected population being delayed for treatment.} Observe that 
$$
\lim_{\varepsilon \rightarrow 0} \dfrac{\lambda I}{1 + \varepsilon I} =  \lambda I \qquad \text{and} \qquad \lim_{\varepsilon \rightarrow +\infty} \dfrac{\lambda I}{1 + \varepsilon I} =  0,
$$
which means that $T_{(\lambda, \varepsilon)}(I)$ is maximum when the saturation of health services is minimum and vice-versa.    

\end{rem}

\begin{rem}
\label{Rem_new_2}
As time evolves, the Infected ($I$) and Recovered ($R$) populations become large and have the same order as $S$. Therefore $I$, $R$ are also impacted by the competition effects and the logistic law in all components would be more suitable to model the reality. Hence, the model \eqref{modelo} under consideration is not biologically realistic.

\end{rem}

\subsection{Hypotheses}
We assume the following conditions, natural in periodically-forced epidemiological contexts:

\smallskip

\begin{itemize}
\item[\textbf{(C1)}]  All parameters are nonnegative; \\ 
\item[\textbf{(C2)}] For all $t \in \RR_0^+$, $S(t)\leq A$; \\
\item[\textbf{(C3)}] For $T>0$ and $\gamma \geq 0$, the map $\Phi: \RR\rightarrow \RR^+$ is $C^3$,  $T$-periodic, $\dpt \dfrac{1}{T}\int_0^T \beta_{\gamma}(t)\, \mathrm{d}t >0 $ and has (at least) two nondegenerate critical points.\\
\end{itemize}

\noindent The phase space associated to \eqref{modelo} is a subset of $(\mathbb{R}_0^+)^3$, induced with the usual topology, and the set of parameters is given by:

$$  {{\Lambda} = \left\{ (A,r,\beta_0,a,\mu,d) \in (\mathbb{R}^+)^6 \right\}, \qquad \gamma\in [0,\varepsilon] \qquad \text{and}\qquad  \omega \in \mathbb{R}^+.}$$

\bigskip

\noindent The parameters $\gamma$ and $\omega$ are not included in $\Lambda$ because they will play a particular role in the emergence of strange attractors in Section \ref{SA_lab}. 

\begin{rem} The variables $S, I, R$ may be interpreted as proportions over the size of the population $N(t)=S(t)+I(t)+R(t)>0$. In our numerics  {(Figures \ref{DFEpoint} and \ref{attr_per_sol}),  the variables and parameter values (unrelated with the reality), identified with the superscript $\sim$, are associated to equation \eqref{modelo} after the following change of variables}:

$$\tilde{S} \mapsto \frac{S}{N}, \qquad \tilde{I} \mapsto \frac{I}{N}\qquad \text{and}\qquad \tilde{R} \mapsto \frac{R}{N}.$$
\end{rem}

The first two equations of (\ref{modelo}), $\dot{S}$ and $\dot{I}$, are independent of  $\dot{R}$. This is why we may reduce (\ref{modelo}) to:

\smallskip

\begin{equation}
\label{modelo2a}
\begin{array}{lcl}
\dot{x} = f_\gamma(x) \quad \Leftrightarrow \quad
\begin{cases}
&\dot{S} =  S(A-S) - \beta_\gamma(t) IS\\
\\
&\dot{I} =  \beta_\gamma(t)IS - (\mu+d) I  - \dfrac{r I}{a + I},
\end{cases}
\end{array}
\end{equation}

\smallskip

\noindent with $x=(S,I)$.

\begin{rem} \label{notacao_d}
 From now on, with the exception of Lemma \ref{compact1}, for the sake of simplicity,  the parameter $\mu$ encloses natural death rate $\mu$ and death rate  due to the disease $d$. In other terms:
$$
\mu +d \mapsto \mu.
$$
\end{rem}

\subsection{Autonomous case ($\gamma=0$)}
The vector field   $f_0(x)$ associated to \eqref{modelo2a} is autonomous, $C^\infty$ and defined on  $(\mathbb{R}_0^+)^2$.
For $\gamma=0$, the model (\ref{modelo2a}) may be recast into the form

\begin{equation}
\label{modelo2}
\begin{array}{lcl}
\dot{x} = f_0(x) \quad \Leftrightarrow \quad
\begin{cases}
&\dot{S} =  S(A-S) - \beta_0IS\\
\\
&\dot{I} =  \beta_0IS -  \mu I - \dfrac{r I}{a + I}.
\end{cases}
\end{array}
\end{equation}

\bigskip

In Lemma \ref{compact1}, we prove that the flow associated to \eqref{modelo2} may be defined in a compact set of $(\mathbb{R}_0^+)^2$, leading to a \emph{complete flow}, \emph{i.e.} solutions are defined for all $t \in \mathbb{R}^+$. The quantity $\mathcal{R}_0$  can be seen as the average number of infectious contacts of a single infected individual during the entire period they remain infectious \cite{Jones2007}.  According to \cite{ParkBolker2020, Li2011}, for model  (\ref{modelo}), this number may be explicitly computed as:
 
\begin{equation}
\label{R0}
\begin{array}{lcl}
\mathcal{R}_0 =\dpt  \lim_{T\rightarrow +\infty}\frac{1}{T} \, \int_0^T\,  \dfrac{A \beta(t)}{\mu + \frac{r}{a}} \, \, \mathrm{d}t \overset{(\gamma=0)}{=} \dfrac{\beta_0A}{\mu + \frac{r}{a}} \geq 0.
\end{array}
\end{equation}

\bigskip

\noindent  Our first main result shows the existence of a non-empty open subset   of $\Lambda$  such that the model \eqref{modelo2} has $\mathcal{R}_0<1$ and  exhibits two endemic equilibria.
\begin{maintheorem}
\label{th: mainA}
There is a non-empty open set \,$\mathcal{U}_1\subset \Lambda$ for which \eqref{modelo2} has $\mathcal{R}_0<1$ and the flow exhibits two endemic equilibria, a sink and a saddle.
\end{maintheorem}

\noindent The proof of Theorem \ref{th: mainA} is  presented in Subsection \ref{proof Th A}, by exhibiting an open set \,$\mathcal{U}_1\subset \Lambda$ where disease-free and endemic equilibria coexist. 
The sink undergoes a supercritical Hopf bifurcation giving rise to an attracting  periodic solution, which survives for $\mathcal{U}_2\subset \Lambda$. This is the purpose of the next result whose proof is performed in Subsection \ref{proof prop1}.

\begin{prop}
\label{prop1}
There is a non-empty open set \,$\mathcal{U}_2\subset \Lambda$ for which \eqref{modelo2} has $\mathcal{R}_0<1$ and the flow exhibits an attracting periodic solution.  
\end{prop}

\noindent The existence of an orientable stable periodic solution  (see Figure \ref{attr_per_sol}) coming from a supercritical Hopf bifurcation, prompts the existence of a \emph{strange attractor} for $f_\gamma$, with $\gamma>0$. The formal statement of this result is the goal of next subsection.
\medbreak
\medbreak
\subsection{The non-autonomous case ($\gamma>0$)}
Many aspects contribute to the richness and complexity of a dynamical system. One of them is the existence of strange attractors (observable chaos). Before going further, we introduce the following notion:

\begin{defn}[\cite{Rodrigues2020}, adapted]
A  (H\'enon-type) \emph{strange attractor} of a two-dimensional dissipative diffeomorphism, defined on a Riemannian manifold, is a compact invariant set $\Omega$ with the following properties:

\smallskip

\begin{enumerate}
\item the set $\Omega$  equals the topological closure of the unstable manifold of a hyperbolic periodic point;
\medskip
\item the basin of attraction of $\Omega$  contains a non-empty open set ($\Rightarrow$ it has positive Lebesgue measure);
\medskip
\item there is a dense orbit in $\Omega$ with a positive Lyapunov exponent.
\end{enumerate}

\bigskip

\noindent A vector field possesses a \emph{strange attractor }if the first return map to a cross section does.  
\end{defn}

The next result  is about a mechanism for producing chaos which may be applied   to some tamed dynamical settings, such as limit cycles and singularities undergoing supercritical Hopf bifurcations. It proves the appearance of sustainable chaotic behavior under reasonable conditions.

\begin{maintheorem}
\label{th: mainB}
For \,$\mathcal{U}_2\subset \Lambda$ of Proposition \ref{prop1} and  for $\omega$ sufficiently large ($\omega\gg 1$),  the following inequality holds for system \eqref{modelo2a}:
\begin{equation}
\label{abundance}
\liminf_{\varepsilon\rightarrow 0^+}\, \,  \frac{ \emph{Leb} \left\{\gamma \in [0,\varepsilon]: {f}_\gamma  \text{  exhibits a strange attractor}\right\}}{\varepsilon}  >0,
\end{equation}

\medskip

\noindent where \emph{Leb} denotes the one-dimensional Lebesgue measure.
\end{maintheorem}

\noindent This result implies that  strange attractors are \emph{abundant} (near $\gamma=0$) for the one-parameter family $\mathcal{F}_\gamma$ associated to the modified SIR model \eqref{modelo} {in the terminology of \cite{MoraViana1993}.}  
The proof of Theorem \ref{th: mainB} is  presented in Section \ref{SA_lab}. Although the proof is highly specialized, its  consequences are  discussed in Section \ref{discussion}. 
Our technique may be applied to all models displaying Bogdanov-Takens bifurcations \cite{Yagasaki2002} with a supercritical Hopf bifurcation line. 

Throughout this paper, we have endeavoured to make a self contained exposition bringing together all topics related to the proofs. We have drawn illustrative figures to make the paper easily readable.

\section{The isolating compact set}
\label{compact}

In this section, we are going to consider system  \eqref{modelo} subject to the condition $\gamma=0$.

\begin{defn} 
We say that $\mathcal{K} \subset (\RR^+_0)^3$ is a \emph{positively flow-invariant set} for \eqref{modelo} if for all $X \in \mathcal{K}$ the trajectory of $\varphi(t, X)$ is fully contained in $\mathcal{K}$ for $t \geq 0$.

\end{defn}

\begin{lem}
\label{compact1}

The region defined by:

$$
\mathcal{M} = \left\{ (S,I,R) \in  (\mathbb{R}_0^+): \quad  0 \leq S \leq A, \quad 0 \leq S+I+R \leq \dfrac{A (\mu + A)}{\mu}, \quad I, R\geq 0 \right\},
$$

\bigskip

\noindent is positively flow-invariant for model (\ref{modelo}) with $\gamma=0$.
\end{lem}

\begin{proof}
It is easy to check that $(\mathbb{R}_0^+)^3$ is flow invariant. We show that if $(S_0, I_0, R_0)\in (\mathbb{R}_0^+)^3$, then $\varphi_0(t, (S_0, I_0, R_0))$, $t\in \RR_0^+$, is contained in $ \mathcal{M}$.
Let us define  $N(t) = S(t) + I(t) + R(t)$ associated to the trajectory $\varphi_0(t, (S_0, I_0, R_0))$.  Using the components of \eqref{modelo} with $\gamma=0$, one knows that:

\smallskip

\begin{equation}
\label{proof1}
\begin{array}{lcl}
\dot{N} & = & \dot{S} + \dot{I} + \dot{R} \\
\\
& = & S(A-S) - \beta_0IS + \beta_0IS - \mu I - dI - \dfrac{r I}{a + I} + \dfrac{r I}{a + I} - \mu R \nonumber \\
\\
& = & S(A-S) - \mu I - dI - \mu R ,
\end{array}
\end{equation}

\bigskip

\noindent from where we deduce that:

\smallskip

\begin{equation}
\label{proof2}
\begin{array}{lcl}
\dot{N} + \mu N & = & S(A-S) - \mu I - dI - \mu R + \mu S + \mu I + \mu R \nonumber \\
\\
& = & SA - S^2 - dI + \mu S  \\ \\
& \leq  &   (\mu + A) S.
\end{array}
\end{equation}

\bigskip

\noindent If $\beta_0=0$, then the first component of \eqref{modelo} would be the logistic growth and thus its solution is limited by $A$ (by \textbf{(C2)}), a property which remains for $\beta_0>0$. In particular, we may conclude that 

\begin{equation*}
\label{proof3}
\begin{array}{lcl}
\dot{N} + \mu N  \leq (\mu + A) A.
\end{array}
\end{equation*}

\bigskip

Multiplying the integrant factor  $a(t)>0$ in   both sides, we obtain

\begin{equation*}
\label{proof4}
\begin{array}{lcl}
\dot{N} a(t) + \mu N a(t) \leq (\mu + A) A a(t) ,
\end{array}
\end{equation*}

\bigskip

\noindent where it is assumed that $\mu a(t) = \dot{a}(t)$, resulting in $a(t) = C_1 e^{\mu t}$, with $C_1>0$. It follows straightforwardly that:

\begin{equation*}
\label{proof5}
\begin{array}{lcl}
&&\dfrac{\mathrm{d}}{\mathrm{d}t} \Big[ N(t) C_1 e^{\mu t} \Big] = (\mu + A) A C_1 e^{\mu t} \\
\\
\Leftrightarrow&& N(t) C_1 e^{\mu t}  =  \dfrac{(\mu + A)A C_1 e^{\mu t}}{\mu} + C_2, \quad C_2 \in \mathbb{R} \\
\\
\Leftrightarrow&& N(t)  =  \dfrac{(\mu + A)A}{\mu} + C_3 e^{-\mu t}, \qquad \text{for} \quad C_3 = \dfrac{C_2}{C_1}.
\end{array}
\end{equation*}

\bigskip

\noindent Since $\dpt \lim_{t\rightarrow +\infty} N(t)=\dfrac{(\mu + A)A}{\mu} $, it is then proved that $N(t)$ is bounded and consequently all the solutions of the model (\ref{modelo}) are equally bounded.
\end{proof}
As pointed out before, from now on, the parameter $\mu$ denotes the natural death rate   and the death rate  due to the disease.  

\section{Disease-free equilibria and stability}
\label{DFE_lab}

In this section, we compute the disease-free equilibria of  (\ref{modelo2}) and their Lyapunov stability.  The model (\ref{modelo2}) has two disease-free equilibria: $$E_1= (0,0) \qquad \text{and} \qquad E_2=(A,0).$$ These equilibria are those where there is no infection present in the population. The jacobian matrix of the vector field (\ref{modelo2}) at a general point $E = (S, I) \in (\mathbb{R}_0^+)^2$ is given by:
 
\begin{equation}
\label{jacob}
\begin{array}{lcl}
J(E)=\left(\begin{array}{cc}
-\beta_0I + A - 2S & -\beta_0S \\ 
\\
\beta_0I & \beta_0S-\mu-\dfrac{ra}{(a+I)^2}
\end{array}\right)
\end{array}.
\end{equation}

\bigskip

\noindent At the disease-free equilibria, $E_1$ and $E_2$, the matrix \eqref{jacob} takes the forms:

\bigskip

\begin{center}
$
J(E_1)=\left(\begin{array}{cc}
A  & 0 \\ 
\\
0 & -\mu-\dfrac{r}{a}
\end{array}\right)$ \quad \text{and}\quad
$
J(E_2)=\left(\begin{array}{cc}
-A  & -\beta_0A \\ 
\\
0 & \beta_0A-\mu-\dfrac{r}{a}
\end{array}\right) .
$
\end{center}

\pgfplotsset{compat = newest}
\usetikzlibrary{decorations.markings}

\begin{figure}[!]
\center
\includegraphics[scale=1]{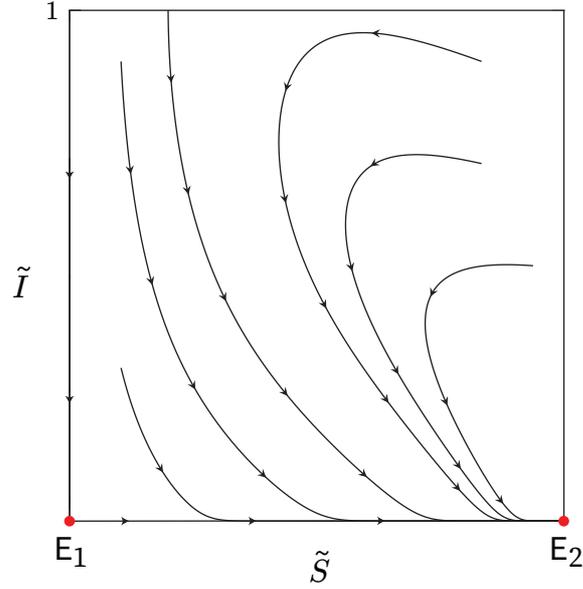}
\caption{\small Phase portrait of  \eqref{modelo2} with $\tilde{A}=0.96$, $\tilde{a}=0.02$, $\tilde{r}=0.25$, $\tilde{\beta_0} = 0.8$ and $\tilde{\mu} = 0.2$. Different trajectories associated to strategic initial conditions have been plotted, as well as the disease-free equilibria $E_1 = (0,0)$ and $E_2 = (A,0)$. Arrows indicate the direction of the flow according to time.}
\label{DFEpoint}
\end{figure}

\begin{lem} \label{DFE_stable}
With respect to system \eqref{modelo2},   $E_1$ is a saddle and $E_2$  is a sink  if and only if $\mathcal{R}_0 < 1$.
\end{lem}

\begin{proof}
The eigenvalues of $J(E_1)$ are $ A>0$ and $ -\mu - \frac{r}{a}<0$. As they are  real with  different signs,  then $E_1$ is a saddle \cite[pp.~4, 8--10]{GuckenheimerHolmes1983}. The eigenvalues of $J(E_2)$ are  $ -A<0$ and $  \beta_0A-\mu-\frac{r}{a}$. Since 

\begin{equation*}
\label{E2A0}
\begin{array}{rcr}
\beta_0A-\mu-\dfrac{r}{a} &<& 0  \quad \Leftrightarrow \quad \dfrac{\beta_0A}{\mu + \frac{r}{a}} \,\,\, \overset{(\ref{R0})}{=} \,\,\, \mathcal{R}_0 <   1 ,
\end{array}
\end{equation*}
\noindent the result follows \cite[pp.~4, 8--10]{GuckenheimerHolmes1983}.
\end{proof}
The two equilibria of Lemma \ref{DFE_stable} and the dynamics nearby have been drawn in Figure~\ref{DFEpoint}.

\section{Endemic equilibria}
\label{END_lab}

We compute the endemic equilibria of  \eqref{modelo2} and we analyze their stability as well as the bifurcations they undergo. This will be used to prove Theorem \ref{th: mainA} in Subsection \ref{proof Th A}. For the sake of completeness, we also perform a {\emph{sensitivity analysis}} of the parameters in Subsection \ref{PHI_0_analysis}. 
 
\subsection{Explicit expression}
In this section, we compute the endemic equilibria by founding non trivial zeros of $f_0$ (see \eqref{modelo2}):

\begin{equation}
\begin{array}{lcl}
\begin{cases}
& (A-S) - \beta_0I = 0\\
\\
&\beta_0S - \mu  - \dfrac{r }{a + I} = 0 
\end{cases} .
\end{array}
\end{equation}

In particular, we have:

$$
I= \dfrac{A - S}{\beta_0}>0 \qquad \text{and} \qquad \beta_0S - \mu - \dfrac{r}{a+\frac{A-S}{\beta_0}}=0
$$

and therefore

\begin{equation}
\label{paraS2}
\begin{array}{lcr}
&&(\beta_0S - \mu)\left(a + \dfrac{A-S}{\beta_0}\right) - r = 0 \\
\\
\Leftrightarrow &&a\beta_0S + S(A-S) - \mu a - \dfrac{\mu}{\beta_0}(A-S) - r = 0 \\
\\
\Leftrightarrow &&a\beta_0S + SA - S^2 - \mu a - \dfrac{\mu}{\beta_0}A + \dfrac{\mu}{\beta_0}S - r = 0 \nonumber\\
\\
\Leftrightarrow &&S^2 - \left[ a\beta_0 + A + \dfrac{\mu}{\beta_0} \right]S + \left[ \dfrac{\mu}{\beta_0}(a\beta_0 + A) + r \right] = 0. 
\end{array}
\end{equation}

\bigskip

\noindent The last equality is a quadratic polynomial in $S$. Hence, \eqref{modelo2} has  two endemic equilibria  if and only if $\Delta>0$, where:

\begin{eqnarray}
\label{Delta}
\nonumber \Delta &=& \left[ a\beta_0 + A + \dfrac{\mu}{\beta_0} \right]^2 - 4\left[ \dfrac{\mu}{\beta_0} (a\beta_0 + A) + r \right] \\ 
\nonumber \\
\nonumber &=& (a\beta_0 + A )^2 + 2(a\beta_0 + A)\dfrac{\mu}{\beta_0} + \left(\dfrac{\mu}{\beta_0}\right)^2 - 4 \left[ \dfrac{\mu}{\beta_0} (a\beta_0 + A) \right] - 4r \nonumber \\ 
\nonumber \\
&=& \left[a\beta_0 + A - \dfrac{\mu}{\beta_0} \right]^2 - 4r . 
\end{eqnarray}

\noindent The endemic equilibria of model ($\ref{modelo2}$) are explicitly given by:

\begin{equation}
\label{eq_endemicos}
\begin{array}{lcl}
E_3 & = & (S_3, I_3) \,\,\, = \,\,\, \left( \dfrac{a\beta_0 + A + \frac{\mu}{\beta_0} - \sqrt{\Delta}}{2} , \dfrac{A - S_3}{\beta_0}  \right) \\
\\
E_4 & = & (S_4, I_4) \,\,\, = \,\,\, \left( \dfrac{a\beta_0 + A + \frac{\mu}{\beta_0} + \sqrt{\Delta}}{2} , \dfrac{A - S_4}{\beta_0}  \right) ,
\end{array}
\end{equation}

\smallskip

\noindent  {where $S_3 < S_4<A$ (by {\bf (C2)}).}  

\begin{rem}
Using (\ref{R0}), the expression for $\Delta$ as a function of $\mathcal{R}_0$ may be written as:
\begin{equation}
\label{DeltaR0}
\begin{array}{lcl}
\Delta =  \left[ a\beta_0 + \mathcal{R}_0\left(\dfrac{\mu}{\beta_0} + \dfrac{r}{a\beta_0} \right) - \dfrac{\mu}{\beta_0} \right]^2 - 4r.
\end{array}
\end{equation}
\end{rem}

\subsection{Preliminary result for the study of bifurcations} \label{phi0phi0phi0<1}
From now on, we settle  the following constants that will be used throughout this text:

\begin{equation}
\label{phi_0}
\begin{array}{lcl}
\phi_0 := 1 - \dfrac{(a\beta_0 - \sqrt{r})^2}{a\mu +r} \qquad \text{and} \qquad \phi_1:=\dfrac{a^2\beta_0^2 + a\mu}{a\mu + r}.
\end{array}
\end{equation}

\noindent It is easy to check that 

\begin{equation*}
\label{phi_0_1}
\begin{array}{lcl}
\phi_0 = 1 - \dfrac{a^2\beta_0^2 + {r}}{a\mu +r} + \dfrac{2a\beta_0\sqrt{r}}{a\mu +r} , \qquad \phi_0 \leq 1 \qquad \text{and} \qquad \phi_1 \geq 0.
\end{array}
\end{equation*}

\smallskip
\noindent The following result relates the existence of endemic  {equilibria} of  \eqref{modelo2} with $\mathcal{R}_0$.
\begin{lem}\label{phi0}
System \eqref{modelo2} has two endemic equilibria if \,$\mathcal{R}_0 > \phi_0$.
\end{lem}

\begin{proof}
We  know that  system (\ref{modelo2}) has two endemic equilibria if and only if $\Delta > 0$. Indeed, 

\begin{eqnarray}
\label{binomio0}
\nonumber && \Delta > 0 \\
\nonumber \\
\nonumber \overset{\eqref{DeltaR0}}{\Leftrightarrow} && \left[ a\beta_0 + \mathcal{R}_0\left(\dfrac{\mu}{\beta_0} + \dfrac{r}{a\beta_0} \right) - \dfrac{\mu}{\beta_0} \right]^2 - 4r > 0 \\
\nonumber \\
\nonumber \Leftrightarrow && a\beta_0 + \mathcal{R}_0 \left( \dfrac{\mu}{\beta_0} + \dfrac{r}{a\beta_0} \right) - \dfrac{\mu}{\beta_0} >  2\sqrt{r} \qquad \vee \qquad a\beta_0 + A - \dfrac{\mu}{\beta_0} < - 2\sqrt{r}  \\
\nonumber \\
\nonumber \Leftrightarrow && \mathcal{R}_0 > \dfrac{2a\beta_0\sqrt{r}}{a\mu + r} - \dfrac{a^2 \beta_0^2}{a\mu + r} + \dfrac{a \mu}{a \mu + r} \quad \,\,\, \vee \qquad a\beta_0 + A - \dfrac{\mu}{\beta_0} < - 2\sqrt{r} \\
\nonumber \\
\nonumber \Leftrightarrow && \mathcal{R}_0 > \dfrac{2a\beta_0\sqrt{r}}{a\mu + r} - \dfrac{a^2 \beta_0^2 + r}{a\mu + r} + 1 \,\,\,\,\,\,\, \qquad \vee \qquad a\beta_0 + A - \dfrac{\mu}{\beta_0} < - 2\sqrt{r} \\
\nonumber \\
\nonumber \Leftrightarrow && \mathcal{R}_0 > 1 - \dfrac{(a\beta_0 - \sqrt{r})^2}{a\mu + r} \qquad \qquad \qquad \,\, \vee \qquad a\beta_0 + A - \dfrac{\mu}{\beta_0} < - 2\sqrt{r} \\
\nonumber \\
\Leftrightarrow && \mathcal{R}_0 > \phi_0 \qquad \qquad \qquad \qquad \qquad \qquad \vee \qquad a\beta_0 + A - \dfrac{\mu}{\beta_0} < - 2\sqrt{r} .
\end{eqnarray}

\smallskip

\noindent Noticing that condition $a\beta_0 + A - \dfrac{\mu}{\beta_0} < - 2\sqrt{r}$ is never satisfied (\emph{cf.} Remark after the proof of Lemma \ref{phi0Positive}), the result follows.
\end{proof}

\begin{lem} \label{phi0Positive}
The following assertions are true:

\smallskip

\begin{enumerate}
\item The condition $a\beta_0 - 2\sqrt{r} < \dfrac{\mu}{\beta_0}$ is satisfied if and only if  $\phi_0 > 0$.\\
\item $ \mathcal{R}_0 >\phi_1$.\\
\item If \,$a\beta_0 < \sqrt{r}$, \,then \,$\phi_1 < 1$.\label{three} \\
\item The condition  $a\beta_0 < \sqrt{r}$ is satisfied if and only if \,$\phi_1 < \phi_0$. \label{four} 
\end{enumerate}
\end{lem}

\smallskip

\begin{proof}
\begin{enumerate}

\item Using (\ref{phi_0}), one may deduce that:

\begin{eqnarray*}
\label{phi_0_2}
\nonumber \phi_0 &=& 1- \dfrac{a^2\beta_0^2+r}{a\mu + r} + \dfrac{2a\beta_0\sqrt{r}}{a\mu + r} \,\,\, = \,\,\, \dfrac{a\mu+r-a^2\beta_0^2-r+2a\beta_0\sqrt{r} }{a\mu + r} = \,\,\,  \dfrac{2\beta_0\sqrt{r} - a\beta_0^2 + \mu}{\mu + \frac{r}{a}} .
\end{eqnarray*}

\smallskip

\noindent Since $\mu + r/a>0$, it follows immediately that:

\begin{equation*}
\label{phi0maior0}
\begin{array}{lcl}
&& \phi_0 > 0 \quad \Leftrightarrow \quad  2\beta_0\sqrt{r} - a \beta_0^2 + \mu > 0 \quad \Leftrightarrow \quad 2\sqrt{r} - a\beta_0 + \dfrac{\mu}{\beta_0} > 0 \quad \Leftrightarrow \quad a\beta_0 - 2\sqrt{r} < \dfrac{\mu}{\beta_0}.
\end{array}
\end{equation*}

 \bigskip
\item Since $ A > S_4 > S_3$ (by {\bf (C2)}), we have:

\begin{equation*}
\label{phi_1}
\begin{array}{lcl}
S_4 = \dfrac{a\beta_0 + A + \frac{\mu}{\beta_0} + \sqrt{\Delta}}{2} < A \quad \Leftrightarrow \quad a\beta_0 + \dfrac{\mu}{\beta_0} + \sqrt{\Delta} < A \quad \Leftrightarrow \quad a\beta_0 - A + \dfrac{\mu}{\beta_0} < - \sqrt{\Delta}.
\end{array}
\end{equation*}

\noindent In particular, we may conclude that:

\begin{eqnarray}
\label{phi_1_2}
& & a\beta_0 + \dfrac{\mu}{\beta_0} - A < 0 \\
\nonumber  \\
\nonumber \overset{(\ref{R0})}{\Leftrightarrow} && a\beta_0 + \dfrac{\mu}{\beta_0} - \dfrac{\mathcal{R}_0 (\mu + \frac{r}{a})}{\beta_0} < 0 \\
\nonumber \\
\nonumber \Leftrightarrow && \mathcal{R}_0 > \dfrac{a\beta_0^2 + \mu}{\mu + \frac{r}{a}} \,\,\, = \,\,\, \dfrac{a^2 \beta_0^2 + a \mu}{a \mu + r} ,
\end{eqnarray}

 \smallskip

\noindent which is equivalent to $\mathcal{R}_0 > \phi_1 > 0$ and the result follows directly. 

  \bigskip
\item One knows that:
\begin{eqnarray}
\label{phi_1_3}
\nonumber \phi_1 & = & \dfrac{a^2 \beta_0^2 + a \mu}{a \mu + r} \,\,\, = \,\,\, \dfrac{a^2 \beta_0^2 + a \mu + r - r}{a \mu + r} \,\,\, = \,\,\, 1 + \dfrac{a^2 \beta_0^2 - r}{a \mu + r} .
\end{eqnarray}

\noindent Since  $a^2\beta_0^2 - r < 0$ (by hypothesis one knows that $a\beta_0 < \sqrt{r}$), we have:

\begin{equation*}
\label{phi_1_4}
\begin{array}{rcl}
a^2 \beta_0^2 - r < 0 \quad \Leftrightarrow \quad a^2 \beta_0^2 < r \quad \Leftrightarrow \quad a\beta_0 < \sqrt{r} ,
\end{array}
\end{equation*}

\noindent and thus   $\phi_1 < 1$.

  \bigskip

\item The proof of this item is a consequence of the following chain of equivalences:
\begin{eqnarray}
\nonumber&& \phi_0 > \phi_1 \\
\nonumber \\
\nonumber &\overset{(\ref{binomio0}), (\ref{phi_1_3})}{\Leftrightarrow} & 1 - \dfrac{(a\beta_0 - \sqrt{r})^2}{a\mu + r} > 1 + \dfrac{a^2 \beta_0^2 - r}{a\mu + r} \\
\nonumber \\
\nonumber &\Leftrightarrow &  - (a\beta_0 - \sqrt{r})^2 > a^2 \beta_0^2 - r \\
\nonumber \\
\nonumber &\Leftrightarrow & - a^2 \beta_0^2 + 2a\beta_0\sqrt{r} - r > a^2 \beta_0^2 - r \\
\nonumber \\
\nonumber &\Leftrightarrow & 2a\beta_0\sqrt{r} > 2 a^2 \beta_0^2 \\
\nonumber \\
&\Leftrightarrow & a\beta_0 < \sqrt{r} . \label{phi_0_1_}
\end{eqnarray}
\end{enumerate}
\end{proof}

\subsection{Digestive remarks} \label{deg_rem}
\begin{rem}
\label{rem_4}
The number $\phi_0$ is the threshold above which we find saddle-node bifurcations and $\phi_1$ is the    $\mathcal{R}_0$-value above which  $S_3<S_4<A$.
\end{rem}
\begin{rem}
\label{rem_imp}
The right hand side inequality of    \eqref{binomio0} is impossible. Using \eqref{phi_1_2}, one knows that $A - \frac{\mu}{\beta_0} > a\beta_0 > 0$ and thus $a\beta_0 + (A - \frac{\mu}{\beta_0})>0$. This would contradict the second inequality of    \eqref{binomio0}.
\end{rem}

\subsection{Saddle-node bifurcation}
\label{ss:sn}
\noindent Let now consider the case where $\mathcal{R}_0 \geq  \phi_0$. If $\mathcal{R}_0 = \phi_0$, then two endemic equilibria are born under the condition $a\beta_0 < \sqrt{r}$, through a saddle-node bifurcation. We address the reader to \cite[pp.~146--149, 157]{GuckenheimerHolmes1983} for more information on the topic. 

\begin{lem}
\label{Lemma SN}
If $\mathcal{R}_0 = \phi_0$ and $\beta_0< 1$, then system (\ref{modelo2}) undergoes a saddle-node bifurcation.
\end{lem}

\begin{proof}
Let $\mathcal{R}_0^{\star} \in [0,1]$ be the $\mathcal{R}_0$ such that $\Delta = 0$ (see (\ref{DeltaR0})), \emph{i.e.}

\begin{equation*}
\label{delta_zero}
\begin{array}{lcl}
&&\left[ a\beta_0 + \mathcal{R}_0 \left(\dfrac{\mu}{\beta_0} + \dfrac{r}{a\beta_0} \right) - \dfrac{\mu}{\beta_0} \right]^2 - 4r = 0 \\
\\
&\overset{\eqref{binomio0}}{\Leftrightarrow} & a\beta_0 + \mathcal{R}_0 \left(\dfrac{\mu}{\beta_0} + \dfrac{r}{a\beta_0} \right) - \dfrac{\mu}{\beta_0} = 2\sqrt{r}  .
\end{array}
\end{equation*}

\bigskip

\noindent Let $A^{\star}$ be the associated $A$-value such that $\Delta=0$. The constant $A^{\star}$ may be calculated in  the following way:

 \begin{eqnarray}
\label{delta_zero_2}
\nonumber && \left[ a\beta_0 + \mathcal{R}_0^{\star} \left(\dfrac{\mu}{\beta_0} + \dfrac{r}{a\beta_0} \right) - \dfrac{\mu}{\beta_0} \right]^2 - 4r = 0 \\
\nonumber \\
\nonumber \Leftrightarrow&& \left[ a\beta_0 + A^{\star} - \dfrac{\mu}{\beta_0} \right]^2 - 4r = 0 \\
\nonumber \\
\nonumber \Leftrightarrow&& a\beta_0 + A^{\star} - \dfrac{\mu}{\beta_0} = 2\sqrt{r} \\
\nonumber \\
\Leftrightarrow&& A^{\star} = 2\sqrt{r} - a\beta_0 - \dfrac{\mu}{\beta_0} .
\end{eqnarray}

\noindent If $\Delta = 0$, then $E_3 \equiv E_4 \equiv E^{\star}$. One knows that $J(E^\star)$ may be written as:

\begin{equation*}
\label{jacobiana_E}
\begin{array}{lcl}
J(E^{\star})=\left(\begin{array}{cc}
-S^{\star} & -\beta_0S^{\star} \\ 
\\
\beta_0I^{\star} & \dfrac{rI^{\star}}{(a+I^{\star})^2}
\end{array}\right)
\end{array} ,
\end{equation*}

\noindent where

\begin{equation*}
\label{S_star}
\begin{array}{lcl}
S^{\star} = \dfrac{A^{\star} + a\beta_0 + \dfrac{\mu}{\beta_0}}{2} 
\qquad \text{and} \qquad I^{\star} = \dfrac{A^{\star} - a\beta_0 - \dfrac{\mu}{\beta_0}}{2\beta_0}.
\end{array}
\end{equation*}

\noindent Using (\ref{delta_zero_2}), we get:

\begin{equation}
\label{jacobiana_E_2}
\begin{array}{lcl}
J(E^{\star})=\left(\begin{array}{cccccc}
-\sqrt{r} - \dfrac{\mu}{\beta_0} &&&&& -\sqrt{r}\beta_0-\mu \\ 
\\
\sqrt{r}-a\beta_0 &&&&& \beta_0(\sqrt{r}-a\beta_0)
\end{array}\right) ,
\end{array} 
\end{equation}

\bigskip

\noindent whose eigenvalues   are $ \sqrt{r}-\frac{\mu}{\beta_0}+\beta_0\sqrt{r}-a\beta_0^2   $ and $0$.  The existence of a zero eigenvalue is a necessary condition for the existence of a saddle-node bifurcation for $f_0$ at $E^{\star}$ \cite[p.~148 (Theorem 3.4.1)]{GuckenheimerHolmes1983}. Instead of checking the nondegeneracy conditions on the nonlinear part of $f_0$ at $E^\star$, we check the emergence of two points of different stability: a sink and a saddle.  Lemma \ref{sink} completes the present proof.
\end{proof}

\begin{lem} \label{sink}
If $\beta_0<1$, then the endemic equilibrium $E_3$ is a sink and $E_4$ is a saddle.
\end{lem}

\begin{proof}
 For $j \in \{3,4 \}$,  the jacobian matrix of $f_0$ at $E_j$ is given by:
 
\begin{equation*}
\label{jacobiana_E3}
\begin{array}{lcl}
J(E_j)=\left(\begin{array}{cc}
-S_j & -\beta_0S_j \\ 
\\
\beta_0I_j & \dfrac{rI_j}{(a+I_j)^2}
\end{array}\right) .
\end{array} 
\end{equation*}

\bigskip

\noindent Let us denote by $\det{J(E_j)}$ and $\mathrm{tr} \,J(E_j)$ the determinant and the trace of $J(E_j)$, respectively. Then we get

\begin{equation}
\label{detE3}
\begin{array}{lcl}
\det{J(E_3)} &=& - \dfrac{r I_3 S_3}{(a+I_3)^2} + \beta_0^2 I_3 S_3 \\
\\
 &=& I_3 S_3 \left( \beta_0^2 - \dfrac{r}{(a + I_3)^2} \right) .
\end{array} 
\end{equation}

\noindent Since $a + I_3 = \dfrac{A+a\beta_0-\frac{\mu}{\beta_0}+\sqrt{\Delta}}{2\beta_0}$ (\emph{cf.} (\ref{eq_endemicos})) \,and \,$\Delta = \left(a\beta_0 + A - \frac{\mu}{\beta_0} \right)^2 - 4r$   (\emph{cf.} (\ref{Delta})), then

\begin{equation*}
\label{detE3_0}
\begin{array}{lcl}
(a + I_3)^2 &=& \dfrac{\left[ a\beta_0 + A - \frac{\mu}{\beta_0} + \sqrt{\Delta} \right]^2 }{4\beta_0^2} \\
\\
&\geq & \dfrac{\left(a\beta_0 + A - \frac{\mu}{\beta_0} \right)^2}{4\beta_0^2} \\
\\
&\overset{\Delta >0}{>}& \dfrac{4r}{4\beta_0^2} = \dfrac{r}{\beta_0^2}.
\end{array} 
\end{equation*}

\noindent Therefore, $\dfrac{r}{(a+I_3)^2} < \beta_0^2$ if and only if $  \beta_0^2 - \dfrac{r}{(a + I_3)^2} > 0$. Coming back to (\ref{detE3}) we get $\det{J(E_3)} > 0.$
 Now we analyze the sign of $\mathrm{tr} \,J(E_3)$:

\begin{equation}
\label{traceE3}
\begin{array}{lcl}
\mathrm{tr} \,J(E_3) = \dfrac{rI_3 - S_3 (a+I_3)^2}{(a+I_3)^2}
\end{array} .
\end{equation}

\bigskip 

\noindent Using the second equation of (\ref{modelo2}) we know that

\begin{equation*}
\label{relations_mod_2}
\begin{array}{lcl}
&& \beta_0S_3I_3 - \mu I_3 - \dfrac{r I_3}{a + I_3}  = 0 \\
\\
\Leftrightarrow && \dfrac{r I_3}{a + I_3} = \beta_0S_3I_3 - \mu I_3 \\
\\
\Leftrightarrow && rI_3 = (a+I_3)(-\mu I_3 + \beta_0S_3I_3) ,
\end{array} 
\end{equation*}

\noindent and replacing it in (\ref{traceE3}), we deduce that

\begin{eqnarray}
\label{finaltraceE3}
\nonumber \mathrm{tr} \,J(E_3) &=&  \dfrac{(a+I_3)(-\mu I_3 + \beta_0S_3I_3) - S_3(a+I_3)^2}{(a + I_3)^2}\\
\nonumber \\
&=& \dfrac{1}{(a + I_3)^2} \big[\left(\beta_0-1\right) S_3 I_3 - \mu I_3 - aS_3 \big] .
\end{eqnarray}

Therefore, if  $\beta_0< 1$, then $\mathrm{tr} \,J(E_3) < 0$ and therefore $E_3$ is a sink. Concerning the equilibrium $E_4$, we get

\begin{equation*}
\label{detE4}
\begin{array}{lcl}
\det{J(E_4)} \,\,\, =\,\,\, - \dfrac{r I_4 S_4}{(a+I_4)^2} + \beta_0^2 I_4 S_4 \,\,\, =\,\,\,  I_4 S_4 \left( \beta_0^2 - \dfrac{r}{(a + I_4)^2} \right) .
\end{array} 
\end{equation*}

\noindent So, it is easy to conclude that if $\beta_0^2 - \dfrac{r}{(a + I_4)^2} < 0$, then $\det{J(E_4)} < 0$ and  $E_4$ is a saddle. Indeed,  the hypothesis is valid due to the following chain of equivalences:

\begin{equation*}
\label{finaldetE4}
\begin{array}{lcl}
&& \beta_0^2 - \dfrac{r}{(a + I_4)^2} < 0 \\
\\
\overset{\eqref{eq_endemicos}}{\Leftrightarrow} && \beta_0^2 < \dfrac{r}{\left( \frac{a\beta_0 + A - \frac{\mu}{\beta_0} - \sqrt{\Delta} }{2\beta_0} \right)^2}
\\
\Leftrightarrow && \left( a\beta_0 + A - \frac{\mu}{\beta_0} - \sqrt{\Delta}\right)^2 < 4r \\
\\
\Leftrightarrow &&  \left( \left( a\beta_0 + A - \frac{\mu}{\beta_0}\right)^2 - 4r \right) - 2\left( a\beta_0 + A - \frac{\mu}{\beta_0}\right)\sqrt{\Delta} + \Delta < 0 \\
\\
\overset{\eqref{Delta}}{\Leftrightarrow} &&  2\Delta - 2\left( a\beta_0 + A - \frac{\mu}{\beta_0}\right)\sqrt{\Delta} < 0 \\
\\
\Leftrightarrow && \Delta < \left( a\beta_0 + A - \frac{\mu}{\beta_0} \right)^2 \\
\\
\overset{\eqref{Delta}}{\Leftrightarrow} && \left( a\beta_0 + A - \frac{\mu}{\beta_0} \right)^2 - 4r < \left( a\beta_0 + A - \frac{\mu}{\beta_0} \right)^2 \\
\\
\Leftrightarrow && -4r < 0.
\end{array} 
\end{equation*}

\end{proof}

\begin{rem}
 The saddle-node bifurcation stated in Lemma \ref{Lemma SN} at $\mathcal{R}_0 = \phi_0$ and $\beta_0<1$ is the dynamical mechanism which explains  that $\mathcal{R}_0 < 1$ is not a sufficient condition to guarantee  {that the $I$-component vanishes} -- this mechanism is called by \emph{backward bifurcation} in \cite{Li2011}.  
\end{rem}

\subsection{Proof of Theorem \ref{th: mainA}}
\label{proof Th A}
Theorem \ref{th: mainA} follows directly from Lemma \ref{sink}  where the open subset $\mathcal{U}_1\subset \Lambda$ is defined by 

\begin{equation*}
a\beta_0<\sqrt{r}, \qquad \phi_0<\dfrac{\beta_0A}{\mu + \frac{r}{a}}= \mathcal{R}_0<1 \qquad \text{and} \qquad \beta_0<1.
\end{equation*}

Recall  that:
\begin{itemize}

\item {the condition $a\beta_0<\sqrt{r}$ forces $\phi_1 <\phi_0<1$ (see Remark \ref{rem_4} and Lemma \ref{phi0Positive}); }

 \medskip
\item  {the term $\phi_0<\mathcal{R}_0<1$ provides the existence of  two equilibria  ($E_3$ and $E_4$);}
\medskip
\item  $E_3$ is a sink provided $\beta_0<1$.
\end{itemize}

\subsection{Sensitivity analysis of $\phi_0$} \label{PHI_0_analysis}

In this subsection we will perform a sensitivity analysis of $\phi_0= 1 - \dfrac{(a\beta_0 - \sqrt{r})^2}{a\mu +r}$ (see \eqref{phi_0}) with respect to the parameters. Assuming $\phi_0$ as a smooth function of  {$a, \beta_0, r$ and $\mu$}, we can make conclusions about the instantaneous progression of the disease depending on the sign of the derivative of $\phi_0$ in order to a fixed parameter.  
In what follows, instead of $\phi_0(a, \beta_0, r, \mu)$ we simply write $\phi_0$.

\begin{lem}
\label{dphi_0}
The following inequalities hold in $\mathcal{U}_1\subset \Lambda$:
$$
\dfrac{\text{d} \phi_0}{\text{d}a} >0, \qquad \dfrac{\text{d} \phi_0}{\text{d}\beta_0} >0, \qquad \dfrac{\text{d} \phi_0}{\text{d}\mu} >0 \qquad \text{and} \qquad \dfrac{\text{d} \phi_0}{\text{d}r} <0.
$$
\end{lem}

\begin{proof}
The proof of this result is straightforward. Indeed, {provided $  \sqrt{r} -a\beta_0> 0$ (see \eqref{phi_0_1_}),} we have:

\begin{equation*}
\label{aaa}
\begin{array}{lcl}
\dfrac{\text{d} \phi_0}{\text{d}a} &=& \dfrac{(-a\beta_0+\sqrt{r}) (a \beta_0 \mu+\sqrt{r} \mu+2 \beta_0 r)}{(a \mu+r)^2} \,\,\, > \,\,\, 0 \\
\\
\dfrac{\text{d} \phi_0}{\text{d}\beta_0} &=& \dfrac{2a(-a\beta_0+\sqrt{r})}{a \mu+r} \,\,\, > \,\,\,0 \\
\\
\dfrac{\text{d} \phi_0}{\text{d}\mu} &=& \dfrac{(a\beta_0-\sqrt{r})^2 a}{(a \mu+r)^2} \,\,\, > \,\,\, 0  \\
\\ 
\dfrac{\text{d} \phi_0}{\text{d}r} &=& \dfrac{a(a\beta_0-\sqrt{r})(\sqrt{r}\beta_0 + \mu)}{(a\mu+r)^2 \sqrt{r}} \,\,\, < \,\,\,0 

\end{array} 
\end{equation*}
\end{proof}

 From Lemma \ref{dphi_0}, taking into account that $\phi_0$ is a value of $\mathcal{R}_0$, for the dynamics of \eqref{modelo2} we may infer that:\\

\begin{enumerate}
\item increasing either the delay in response to treatment ($\Leftrightarrow$ saturation of health services increases) or  the death rate implies a growth of $\mathcal{R}_0$;
\\
\item if the transmission rate of the disease (in the absence of seasonality) increases, then $\mathcal{R}_0$ increases as well;
\\
\item Increasing the cure rate decreases the average number of infectious contacts of a single infected individual during the entire period they remain infectious.
\end{enumerate}

\subsection{Hopf bifurcation}
\label{APS}
In this subsection, we exhibit an open subset $\mathcal{U}_2$ of $\Lambda$ where the equilibrium $E_3$ undergoes a Hopf bifurcation generating an attracting (orientable) non trivial $T$--periodic solution, say $\mathcal{C}$. We address the reader to \cite[{pp.~150--156}]{GuckenheimerHolmes1983} for more information about the topic.  {From now on, $\mathbf{H}$ denotes  the $\mathcal{R}_0$--parameter for which the autonomous system (\ref{modelo2}) exhibits a Hopf bifurcation.}  {We remind the reader that $\mathcal{R}_0 = \frac{\beta_0A}{\mu + \frac{r}{a}}$.}

\begin{lem} \label{Hopf_Bifurcation}
If $\beta_0 > 1$ and $\mu \leq A(\beta_0-1) + a\beta_0 - 2 \sqrt{\beta_0(\beta_0-1)aA}$, then $E_3$ undergoes a supercritical Hopf bifurcation.
\end{lem}

\begin{proof} \label{ProofHopfBif}
The Hopf bifurcation exists when $\mathrm{tr} \,J(E_3) = 0$ and we know from (\ref{eq_endemicos}) that $S_3 = A - \beta_0 I_3$, so
\begin{eqnarray}
\label{TrE30}
&&\nonumber \mathrm{tr} \,J(E_3) = 0 \\
\nonumber \\
\nonumber \overset{\eqref{finaltraceE3}}{\Leftrightarrow}&& \dfrac{1}{(a + I_3)^2} \big[\left(\beta_0-1\right) S_3 I_3 - \mu I_3 - aS_3 \big]=0 \\
\nonumber \\
\nonumber \Leftrightarrow&& (\beta_0-1)S_3I_3 - aS_3 - \mu I_3 = 0 \\
\nonumber \\
\nonumber \Leftrightarrow&& (\beta_0-1)(A-\beta_0 I_3) I_3 - a(A-\beta_0I_3) - \mu I_3 = 0 \\
\nonumber \\
\nonumber \Leftrightarrow&& (\beta_0-1)AI_3 - \beta_0(\beta_0-1){I_3}^2 - aA + a\beta_0I_3 - \mu I_3 = 0 \\
\nonumber \\
\Leftrightarrow&& - \beta_0 (\beta_0-1){I_3}^2 + \big[ (\beta_0-1)A + a\beta_0 - \mu \big] I_3 - aA = 0 .
\end{eqnarray}
 
 The expression \eqref{TrE30} may be seen as a quadratic expression of $I_3$ with discriminant  $\Delta_2$. Hence:
 
\begin{eqnarray*}
\label{um}
\nonumber \Delta_2 &=& \big[ (\beta_0-1)A + a\beta_0 - \mu  \big]^2 - 4\beta_0(\beta_0-1)aA \\
\nonumber \\
&=& \big[ (\beta_0-1)A + a\beta_0 -\mu \big]^2 - \big[ 2\sqrt{\beta_0(\beta_0-1) aA} \big]^2 .
\end{eqnarray*}

 There are two conditions that  {must} be met at a first stage:

\begin{equation}
\label{delta_222}
\begin{array}{lcl}
\Delta_2 \geq  0
\end{array}
\end{equation}

\noindent and

\begin{eqnarray}
\label{dois}
(\beta_0-1)A + a\beta_0 - \mu > 0 \qquad \Leftrightarrow \qquad \mu < (\beta_0-1)A + a\beta_0 .
\end{eqnarray}

\medskip

\noindent {For equality \eqref{TrE30} to hold, the following condition must be satisfied:}

\begin{eqnarray*}
\nonumber \Delta_2 = \big[ (\beta_0-1)A + a\beta_0 -\mu \big]^2 - \big[ 2\sqrt{\beta_0(\beta_0-1) aA} \big]^2 \overset{(\ref{delta_222})}{\geq}  0, 
\end{eqnarray*}

\noindent which is equivalent to
\begin{eqnarray}
\tiny
\label{FromTwo}
\big[  (\beta_0-1)A + a\beta_0 - \mu - 2\sqrt{\beta_0(\beta_0-1) aA} \big] \cdot \big[  (\beta_0-1)A + a\beta_0 - \mu + 2\sqrt{\beta_0(\beta_0-1) aA} \big] \geq 0 .
\end{eqnarray}

\noindent Since $\beta_0>1$, by (\ref{dois}) one knows  that $\big[(\beta_0-1)A + a\beta_0 - \mu + 2\sqrt{\beta_0(\beta_0-1) aA} \big]>0$. Thus, for condition (\ref{FromTwo}) to be verified, we should assume that

\begin{eqnarray}
\label{FromTwo_2}
&&\nonumber \big[(\beta_0-1)A + a\beta_0 - \mu - 2\sqrt{\beta_0(\beta_0-1) aA} \big] \geq 0 \\
\nonumber \\
&\Leftrightarrow &\mu \leq (\beta_0-1)A + a\beta_0 - 2\sqrt{\beta_0(\beta_0-1) aA} .
\end{eqnarray}

\noindent Comparing expressions (\ref{dois}) and (\ref{FromTwo_2}) it is easy to check that

\begin{eqnarray*}
(\beta_0-1)A + a\beta_0 - 2\sqrt{\beta_0(\beta_0-1) aA} < (\beta_0-1)A + a\beta_0 .
\end{eqnarray*}

\noindent Then $\beta_0>1$ and $\mu \leq (\beta_0-1)A + a\beta_0 - 2\sqrt{\beta_0(\beta_0-1) aA}$ are the conditions that meet the requirements for Hopf bifurcation to exist.

 The Hopf bifurcation occurs at points where the map $Df_0(E_3)$ is not the identity, its eigenvalues  have the form $\pm i\omega$ ($\omega>0$) and satisfies the nondegeneracy conditions described in \cite[{pp.~150--156}]{GuckenheimerHolmes1983} on the nonlinear part (namely the variation's speed of the real part of the eigenvalues with respect to the parameters). Instead of verifying these additional conditions, we have checked numerically the emergence of an attracting periodic solution $\mathcal{C}$ in Figure \ref{attr_per_sol} for parameters lying in $\mathcal{U}_2$.
\end{proof}

\begin{figure}[!]
\center
\includegraphics[scale=1]{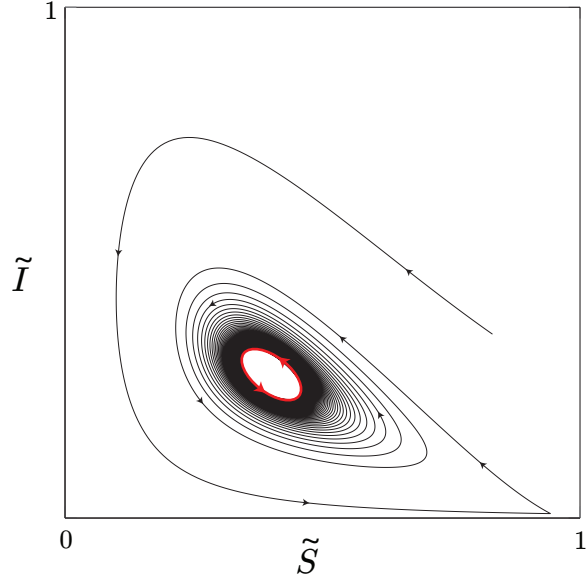}
\caption{\small Phase portrait of system \eqref{modelo2} with $\tilde{A}=0.96$, $\tilde{a}=0.14$, $\tilde{r}=0.25$, $\tilde{\beta_0} = 2$ and $\tilde{\mu} = 0.2$, with initial condition $(\tilde{S}_0,\tilde{I}_0) = (0.8333, 0.3666)$. It is stressed the existence of an attracting periodic solution $\mathcal{C}$. Arrows indicate the flow induced by $t$.}
\label{attr_per_sol}
\end{figure}

\bigskip

\subsection{Proof of Proposition \ref{prop1}}
\label{proof prop1}
Proposition \ref{prop1} is constructed directly from Lemma \ref{Hopf_Bifurcation} of Subsection \ref{APS} where the open subset $\mathcal{U}_2 \subset \Lambda$ is defined by 

\begin{equation*}
\mu \leq A (\beta_0-1) + a\beta_0 - 2 \sqrt{\beta_0(\beta_0-1)aA}, \qquad \beta_0>1 \qquad \text{and} \qquad \phi_0<\dfrac{\beta_0A}{\mu + \frac{r}{a}}<1.
\end{equation*}

\begin{rem}
A Hopf bifurcation of $E_3$ occurs for
$$\beta_0>1 \quad \text{and} \quad \mu \leq A (\beta_0-1) + a\beta_0 - 2 \sqrt{\beta_0(\beta_0-1)aA}.$$
\medbreak
We write \emph{implicitly} the value of $\mathcal{R}_0$ where this bifurcation occurs:
\begin{eqnarray*}
\label{um}
&& \mu \leq A (\beta_0-1) + a\beta_0 - 2 \sqrt{\beta_0(\beta_0-1)aA}  \\
\nonumber \\ 
&\Leftrightarrow & \mu + \frac{r}{a} \leq \frac{r}{a}+ A (\beta_0-1) + a\beta_0 - 2 \sqrt{\beta_0(\beta_0-1)aA} \\ 
&\Leftrightarrow &\frac{1}{ \mu + \frac{r}{a}} \geq \frac{1}{ \frac{r}{a}+ A (\beta_0-1) + a\beta_0 - 2 \sqrt{\beta_0(\beta_0-1)aA}} \\
&\Leftrightarrow &\mathcal{R}_0 \geq \frac{\beta_0 A}{ \frac{r}{a}+ A (\beta_0-1) + a\beta_0 - 2 \sqrt{\beta_0(\beta_0-1)aA}}=:\mathbf{H}. \\ 
\end{eqnarray*}
Any open set around the parameter values used to perform Figure \ref{attr_per_sol}   realizes values of $\mathcal{R}_0$ such that $\phi_0<\mathcal{R}_0=\mathbf{H}<1$.
\end{rem}

\begin{figure}[!]
\center
\includegraphics[scale=0.75]{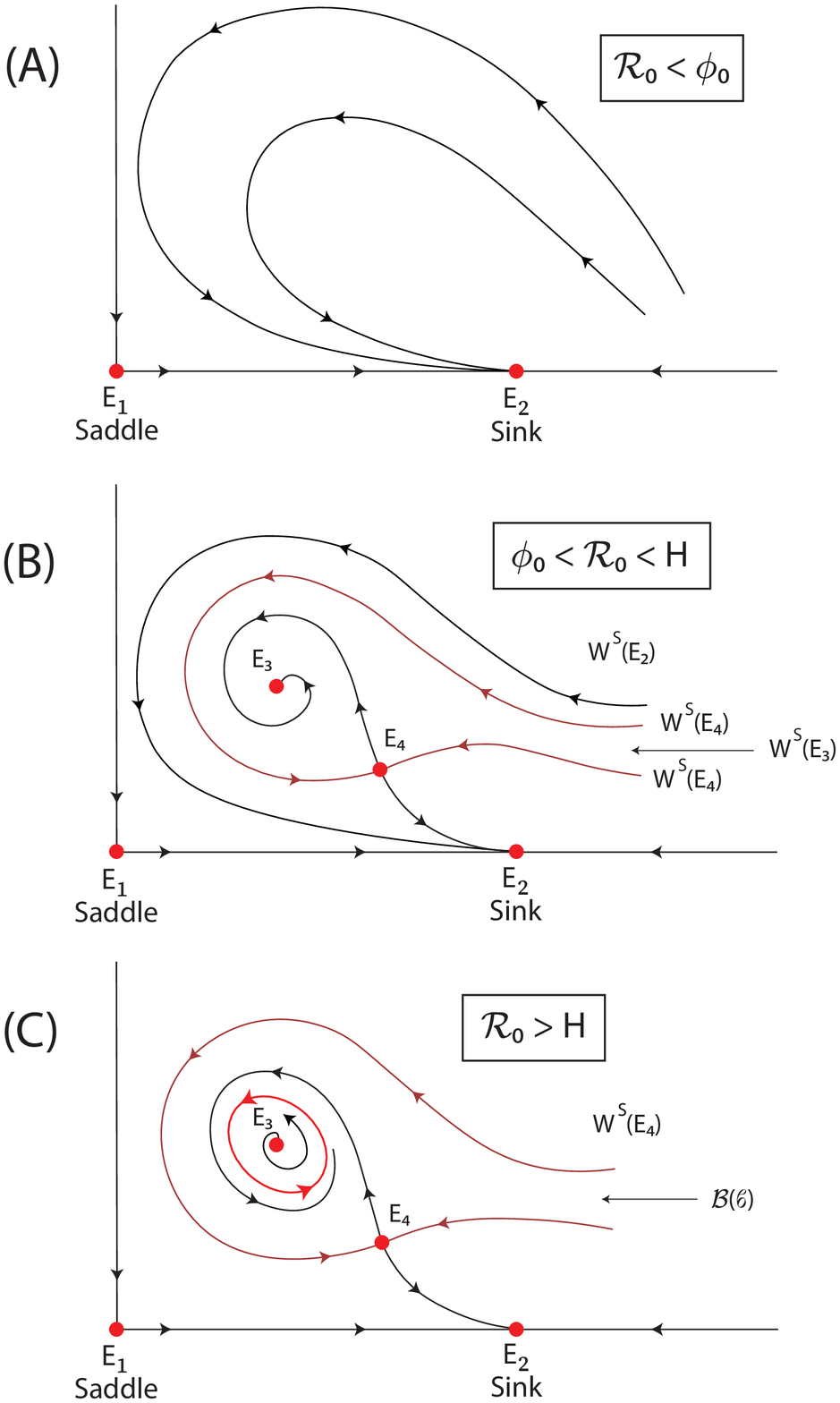}
\caption{Sketch of the phase diagram of \eqref{modelo2} for different values of $\mathcal{R}_0$ and the basin of attraction of the saddles. {\bf (A)} $\mathcal{R}_0<\phi_0$: the disease-free equilibrium $E_2$ is a global attractor. {\bf (B)} $\phi_0<\mathcal{R}_0<\mathbf{H}$: besides the disease-free equilibria $E_1$ and $E_2$, there exists a sink $E_3$ and a saddle $E_4$. {\bf (C)}  {$\mathcal{R}_0>\mathbf{H}$}: there exists an attracting limit cycle. $\mathbf{H}$ represents the Hopf bifurcation and $\mathcal{C}$ is the attracting solution. Remind that $E_1= (0,0)$ and $E_2= (A,0)$.}
\label{esquema}
\end{figure}

\section{The strange attractor}
\label{SA_lab}

In order to prove Theorem \ref{th: mainB}, we are going to make use of the Wang and Young theory on rank-one strange attractors \cite{WangYoung2003}. 
It is a comprehensive chaos theory for a non-uniformly hyperbolic setting that is flexible enough to be applicable to concrete systems
of differential equations and has experienced unprecedented growth in the last 20 years in the context of non-autonomous systems.
\subsection{Proof of Theorem \ref{th: mainB}}
\label{s: 6.1}

The proof follows from our Proposition \ref{prop1} combined with  \cite[Theorems 1 and 2]{WangYoung2003}, taking into account the following considerations:
\begin{itemize}
\item $\omega\gg 1$ by hypothesis;
\medskip
\item the term $\beta_\gamma(t) = \beta_0\left(1+\gamma \Phi(\omega t)\right) > 0$ may be seen as the radial kick of \eqref{modelo2a}; 
\medskip
\item  the non-autonomous periodic forcing of \eqref{modelo2a} is at least $C^3$ and has two nondegenerate critical points (by \textbf{(C3)});
\medskip
\item for \eqref{modelo2a}, the periodic solution $\mathcal{C}$ is attracting and orientable.
\end{itemize}

The abundance of parameters for which we observe strange attractors follows from \cite[{Section 3}]{WangOtt2008}: there exists $\varepsilon>0$ such that for Lebesgue-almost all $\gamma\in [0,\varepsilon]$, the non-wandering set associated to $f_\gamma$ has strange attractors. The chaos is realized for points that belong to the basin of attraction of $\mathcal{C}$ for $\gamma=0$ (see Fig. \ref{esquema}), that will be denoted by $\mathcal{B}(\mathcal{C})$.

 \subsection{Emergence of the strange attractor: a geometric point of view}

Under the notation established in Remark \ref{notacao_d}, in the absence of forcing ($\gamma=0$ in \eqref{modelo2a}), the picture for a supercritical Hopf bifurcation is well known: a stable equilibrium loses its stability when a pair of complex conjugate eigenvalues crosses the imaginary axis, resulting in the appearance of a limit cycle which increases in diameter as it moves away from $E_3$. Subjecting system \eqref{modelo2} to the periodic forcing $\beta_\gamma$, there is a sufficiently large frequency for which one observes  strange attractors.  We now describe the main bifurcations associated to the emergence of observable chaos.
\bigbreak
Considering   $(A,r,\beta_0,a,\mu) \in \mathcal{U}_2$, $\gamma\in [0,\varepsilon]$ and $\omega\in \mathbb{R}^+$, the model \eqref{modelo2a}  may be extended to the three-dimensional system in $\mathbb{R}^2 \times \mathbb{S}^1$, where $\mathbb{S}^1$ is a quotient space:

\smallskip

\begin{equation}
\label{modelo_novo}
\begin{cases}
\begin{array}{lcl}
\dot{S}&=&S(A-S) - \beta_0 \big(1 + \gamma \Phi (\theta)\big)IS\\
\\
\dot{I}&=&\beta_0 \big(1 + \gamma \Phi (\theta)\big)IS - \mu I - \dfrac{r I}{a + I}\\
\\	
\dot{\theta}&=& \omega .
\end{array}
\end{cases}
\end{equation}

\bigskip

\begin{lem}
\label{new_lema}
For $\gamma=0$, $\omega\in \mathbb{R}^+$ and $(A,r,\beta_0,a,\mu) \in \mathcal{U}_2$, the flow of \eqref{modelo_novo} exhibits an attracting 2-dimensional torus $\mathbb{T}$, which is normally hyperbolic.
\end{lem}

\begin{proof}
If $(A,r,\beta_0,a,\mu) \in \mathcal{U}_2$, then the dynamics of \eqref{modelo_novo} restricted to the plane $(S,I)$ has an attracting non trivial periodic solution (\emph{cf.} Proposition \ref{prop1}). Adding the phase component $\dot{\theta} = \omega$, $\omega > 0$, it yields an attracting 2-dimensional torus. Normal hyperbolicity follows from the attractiveness of the torus  \cite{HPS}.
\end{proof}

For  $\omega\in \mathbb{R}^+$ and $(A,r,\beta_0,a,\mu) \in \mathcal{U}_2$, the torus $\mathbb{T}$ of Lemma \ref{new_lema} persists for $\gamma\gtrsim 0$. Let us denote by $\mathbb{T}_\gamma$ its hyperbolic continuation. For $\gamma = 0$ and $\omega\in \mathbb{R}^+$ fixed, take a cross section $\Sigma$ to $\mathbb{T}_0$ in such a way that $\Sigma \cap \mathbb{T}_0$ is a smooth invariant curve $\mathcal{C}$ diffeomorphic to a circle. For $\omega \in \mathbb{R}^+$, at least one of the eigenvalues of $\mathrm{d}\mathcal{G}_{(0,\omega)}|_\mathcal{C}$  has modulus less then 1. 

For $\gamma > 0$ small and $\omega\in \mathbb{R}^+$ fixed, let $\mathcal{G}_{(\gamma, \omega)}$ be the first return map to $\Sigma$ defined in $(\mathcal{B}(\mathcal{C}) \times \mathbb{S}^1)\cap \Sigma$ (basin of attraction of $\mathbb{T}_0$ restricted to $\Sigma$), which is well defined.

\begin{figure}[!]
\center
\includegraphics[scale=0.45]{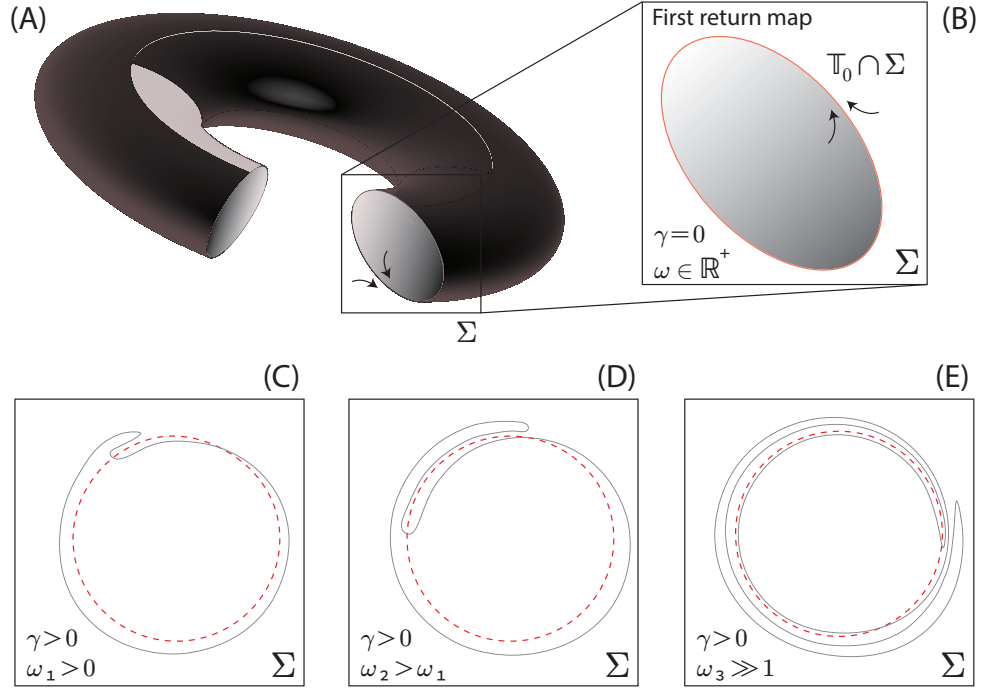}
\caption{\small {\bf (A):} Attracting two-dimensional torus $\mathbb{T}_0$ of system \eqref{modelo_novo} with $\gamma=0$, $\omega\in \mathbb{R}^+$ and $(A,r,\beta_0,a,\mu) \in \mathcal{U}_2$. {\bf (B):} Attracting curve associated to the first return map to a  section $\Sigma$ transverse to $\mathbb{T}_0$ (same parameter conditions as in \textbf{(A)}). {\bf (C)--(E):} Topological horseshoes and emergence of strange attractors for \eqref{modelo_novo} with $\gamma>0$, $0 < \omega_1 < \omega_2 < \omega_3$, $\omega_3 \gg 1$ and $(A,r,\beta_0,a,\mu) \in \mathcal{U}_2$.}
\label{TORO}
\end{figure}

For $\gamma>0$ fixed, if $\omega > 0$ then the attracting torus starts to disintegrate into a finite collection of periodic saddles and sinks, a phenomenon occurring within an ``Arnold tongue'', developing horseshoes (subsets topologically conjugate to a full shift over a finite number of symbols \cite{Passeggi2018, WangYoung2002}), as suggested in Figure \ref{TORO} {\bf (C)}. Once they appear, they persist and  correspond to what the authors of \cite{WangYoung2003} call \emph{transient chaos}.

 As $\omega$ gets larger, the initial deformation on the attracting torus introduced by the perturbing term $\gamma \Phi(\theta)$ is exaggerated further, giving rise to  strange attractors created by stretch-and-fold type actions  -- \emph{sustained chaos}  \cite{WangYoung2003}. The strange attractors contain, but do not coincide with, topological horseshoes.  This is precisely the main difference between our proof and that of \cite{Barrientos2017}.

\subsection{Backward bifurcation}
Following \cite{Li2011}, backward bifurcations occur when multiple stable equilibria
coexist in an epidemic model with $\mathcal{R}_0<1$. 
  If the $I$-component of the initial conditions
is sufficiently small (in its early stage), then trajectories will approach the
disease-free equilibrium $E_2$ and the \emph{Infectious} will be eradicated.
Nevertheless, if the initial conditions are large ($I$-component is large), then the system
will approach the endemic equilibrium $E_3$ and the \emph{Infectious} will
persist.  See Figure \ref{fig_bifurcation} for an illustrative scheme of this description.

\begin{center}
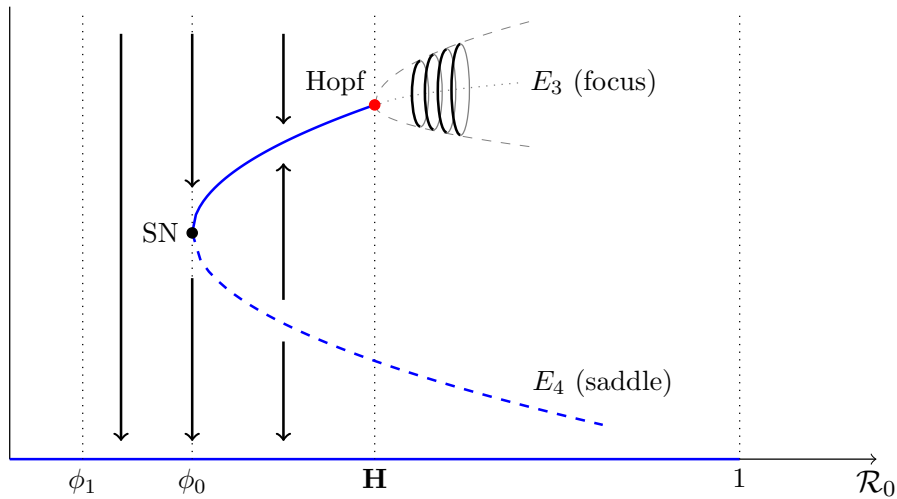
\begin{figure}[ht!]
\scalemath{1.2}{
\begin{tikzpicture}

\draw [-] (-2,5)--(-2,0); 
\draw [->] (-2,0)--(7.5,0) node[below, scale = 0.9]{$\mathcal{R}_0$}; 
\draw [-,blue,thick] (-2,0)--(6,0) node[below, scale = 0.9]{};


\draw [dotted] (-1.2,4.9)--(-1.2,0) node[below, scale = 0.8]{\textcolor{black}{$\phi_1$}}; 

\draw [dotted] (2,4.9)--(2,0) node[below, scale = 0.8]{\textcolor{black}{$\mathbf{H}$}}; 

\draw [<-, thick] (-0.78,0.2)--(-0.78,4.7); 

\draw [dotted] (0,4.9)--(0,0) node[below, scale = 0.8]{\textcolor{black}{$\phi_0$}}; 

\draw [<-, thick] (0,3)--(0,4.7); 
\draw [<-, thick] (0,0.2)--(0,2); 

\draw [->, thick] (1,4.7)--(1,3.7); 
\draw [<-, thick] (1,3.27)--(1,1.76); 
\draw [<-, thick] (1,0.2)--(1,1.3); 

\draw [dotted] (6,4.9)--(6,0) node[below, scale = 0.8]{\textcolor{black}{1}}; 
\draw [blue, thick, domain=0:2, samples=50] plot(\x, {2.5+sqrt(\x)}); 
\draw [blue, dashed, thick, domain=0:4.5, samples=50] plot(\x, {2.5-sqrt(\x)}) node[above=0.5em, black, scale = 0.8] {$E_4$ (saddle)}; 
\filldraw[black] (0,2.5) circle (1.6pt) node[left=0.1em, scale = 0.8] {SN}; 

\draw [gray, very thin, dashed, domain=0:1.7, samples=50] plot(\x+2, {3.91421356237+sqrt(\x/2)}); 
\draw [gray, dotted, domain=0:1.6, samples=50] plot(\x+2, {3.87+sqrt(\x/19)}) node[black, right, scale = 0.8] {$E_3$ (focus)}; 
\draw [gray, very thin, dashed, domain=0:1.7, samples=50] plot(\x+2, {3.91421356237-sqrt(\x/8)}); 

\draw[thick] (2.5,4.399) arc(90:270:2.6pt and 10.28pt);
\draw[gray, thin] (2.5,4.402) arc(90:-90:2.6pt and 10.39pt); 
\draw[thick] (2.645,4.469) arc(90:270:2.7pt and 11.75pt);
\draw[gray, thin] (2.645,4.473) arc(90:-90:2.7pt and 11.90pt);
\draw[thick] (2.79,4.528) arc(90:270:2.8pt and 13.04pt);
\draw[gray, thin] (2.79,4.534) arc(90:-90:2.8pt and 13.2pt); 
\draw[thick] (2.94,4.587) arc(90:270:2.8pt and 14.30pt);
\draw[gray, thin] (2.94,4.587) arc(90:-90:2.8pt and 14.35pt); 

\filldraw[red] (2,3.91421356237) circle (1.6pt) node[above left, scale = 0.8] {\textcolor{black}{Hopf}}; 

\end{tikzpicture}
}
\caption{\small Schematic bifurcations and stability of the endemic equilibria for   model \eqref{modelo2}. The sink $E_3$ undergoes a supercritical Hopf bifurcation (at $\mathbf{H}$) giving rise to an attracting periodic solution. For $\mathcal{R}_0 = \phi_0$, two endemic equilibria are born (saddle-node bifurcation (at SN)), a focus $E_3$ (stable for $\mathcal{R}_0 < \mathbf{H}$ and unstable for $\mathcal{R}_0 > \mathbf{H}$) and a saddle $E_4$. The value $\phi_1$ is the threshold above which $S_3<S_4<A$. Bold arrows indicate the stability of the endemic equilibria.}
\label{fig_bifurcation}
\end{figure}
\end{center}

\section{Discussion and final remarks}
\label{discussion}

In this paper, we analyzed a periodically-forced dynamical system inspired by the SIR endemic model through the addition of a non-autonomous term of the form $$\beta_\gamma(t) = \beta_0\left(1+\gamma \Phi(\omega t)\right).$$ 
As far as we know, this work is the first analytical investigation of the interplay between seasonality, deterministic dynamics and the persistence of strange attractors in this class of biologically inspired models.

\subsection{Results}
We proved that, under particular conditions for the autonomous model \eqref{modelo2} with $\gamma=0$, two endemic equilibria exist for a \emph{basic reproduction number} $\mathcal{R}_0 $ less than 1. More precisely, in Theorem \ref{th: mainA}, we have exhibited an open set in the space of parameters, for which $\mathcal{R}_0 < 1$ and the $I$-component persists in a robust way through two endemic equilibria: one saddle and one sink. For $\mathcal{R}_0 <1$, the sink undergoes a Hopf bifurcation yielding  an attracting periodic solution (Proposition \ref{prop1}). 

For $\mathcal{R}_0 < 1$ and $\omega \gg 1$, using the theory of  strange attractors developed in \cite{WangYoung2003}, we proved in Theorem \ref{th: mainB} that the flow of \eqref{modelo} exhibits abundant strange attractors, meaning that the  {$I$-component does not vanish} and its control may not be  possible. A partial scheme of our conclusions is illustrated in Figure \ref{fig_bifurcation}. For $\gamma > 0$, the numerical description of the dynamics of \eqref{modelo2a} when $\omega$ varies is similar to what has been described in \cite{CastroRodrigues2021}.

Seasonal variations may be captured by introducing periodically-perturbed terms into a deterministic differential equation \cite{Keeling2001}. The periodically-perturbed term $\Phi(t)$ in $\beta_{\gamma}(t)$ may be seen as a ``natural'' $2\pi$-periodic map over the time with two global extrema (governing the high and lower seasons defined by weather conditions). The parameter $\omega$ governs the frequency of $\Phi$, which may be interpreted as a seasonal constraint and varies according to political reasons or scholar holidays, for instance. The complete analysis lies beyond the scope of this article.

\subsection{Strengths and limitations}
Our study has been concentrated in a nonlinear forced model inspired by the problem of modeling infectious diseases. We were not concerned about the validity of the biological value of the model but rather on  the analytical proof of the existence of strange attractors which is, in general, a difficult task.

Our method to find observable chaos (in the terminology of  \cite{WangYoung2002}) cannot be directly used in the classic SIR model \cite{Kermack1932}. It is necessary to include both the logistic growth in \eqref{modelo} in order to define an open set of parameters for which a supercritical Hopf bifurcation happens. Yielding an invariant torus by the lifting process $\dot{\theta}=\omega$, Torus-breakdown theory may be applied and the abundance of strange attractors follows.  The proof of  Theorem \ref{th: mainB} is valid for all models undergoing Bogdanov-Takens bifurcations with a curve of supercritical Hopf bifurcations.

\subsection{Literature}
In a similar model, using the theory developed in \cite{Medio2009, HerreraZanolin2014}, the authors of \cite{Barrientos2017} have  proved the existence of chaotic dynamics (hyperbolic topological \emph{horseshoes}),  not necessarily observable in numerics, through   a step function.  Our contribution goes further since we have been able to prove the existence of \emph{persistent strange attractors}.  Moreover, our results are consistent with the empirical belief that intense seasonality induces chaos \cite{Keeling2001,Barrientos2017,Duarte2021}.

The authors of \cite{Perez2019} studied an adapted SIR model where the incidence rate  was constant instead of a periodic map.  Their main goal was to prove Bogdanov-Takens bifurcations. Since these phenomena have codimension 2, they are difficult to be found. Our goal is different: we relate  $\mathcal{R}_0<1$ with the  persistence of the $I$-component (autonomous case) and the existence of sustainable chaos (periodically-perturbed case) -- see also \cite{Barrientos2017,Li2011}.

\subsection{Future work: the vaccination}
The existence of  persistent strange attractors may be seen  as an undesirable phenomenon associated to unpredictability. As
a consequence, the problem of converting chaos into regular motions becomes
particularly relevant.  In the context of system \eqref{modelo},  avoid stochastic dynamics might be performed by the introduction of a periodically-perturbed term modeling a seasonal vaccination strategy, say $v(t)$. Numerical simulations of \cite{Duarte2021} show that the phase difference
between the two periodic functions (contact rate $\beta(t)$ and vaccination $v(t)$) might play an important role in controlling chaos.  
The rationale for the vaccination policy is to ensure that the proportion of susceptive individuals would stay below a given threshold. 

We guess that if the frequency of $v(t)$ is sufficiently close to the frequency of $\beta(t)$ and separated by a phase constant, then strange attractors are no longer possible and we may stabilize the dynamics.  The proof of this conjecture is an ongoing work.

\section*{Acknowledgements}
The  authors are grateful to the three reviewers for the corrections and suggestions which helped to improve the readability of this manuscript.


\begin{thebibliography}{777}
\bibliographystyle{unsrt}

\bibitem{Carvalho2020} 
T. de Carvalho, R. Cristiano, L.F. Gon\c{c}alves, D.J. Tonon, Global analysis of the dynamics of a mathematical model to intermittent HIV treatment, Nonlinear Dyn. 101 (2020) 719--739. 

\bibitem{Bonyah2020} 
E. Bonyah, F. Al Basir, S. Ray, Hopf Bifurcation in a Mathematical Model of Tuberculosis with Delay, in: P. Manchanda, R. Lozi, A. Siddiqi (Eds.), Mathematical Modelling, Optimization, Analytic and Numerical Solutions. Industrial and Applied Mathematics. Springer, Singapore, 2020, pp. 301--311.

\bibitem{Rajagopal2020} 
K. Rajagopal, N. Hasanzadeh, F. Parastesh, I.I. Hamarash, S. Jafari, I. Hussain, A fractional-order model for the novel coronavirus (COVID-19) outbreak, Nonlinear Dyn. 101 (2020) 711--718. 

\bibitem{Cobey2020} 
S. Cobey, Modeling infectious disease dynamics, Science 368 (2020) 713--714. 

\bibitem{Britton2010} 
T. Britton, Stochastic epidemic models: A survey, Math. Biosci. 225 (2010) 24--35.

\bibitem{Kermack1932} 
W.O. Kermack, A.G. McKendrick, Contributions to the mathematical theory of epidemics. II. -- The problem of endemicity, Proc. R. Soc. Lond. 138 (1932) 55--83.

\bibitem{Dietz1976} 
K. Dietz, The incidence of infectious diseases under the influence of seasonal fluctuations, in: J. Berger, W.J. B\"{u}hler, R. Repges, P. Tautu (Eds.), Mathematical models in medicine. Lecture Notes in Biomathematics, vol 11. Springer, Berlin, Heidelberg, 1976, pp. 1--15.

\bibitem{ParkBolker2020} 
S.W. Park, B.M. Bolker, A Note on Observation Processes in Epidemic Models, Bull. Math. Biol. 82 (2020) 8 pages

\bibitem{Keeling2001}  
M.J. Keeling, P. Rohani, B.T. Grenfell, Seasonally forced disease dynamics explored as switching between attractors, Physica D 148 (2001) 317--335.

\bibitem{Buonomo2018} 
B. Buonomo, N. Chitnis, A. d'Onofrio, Seasonality in epidemic models: a literature review, Ric. di Mat. 67 (2018) 7--25.

\bibitem{Moghadami2017} 
M. Moghadami, A Narrative Review of Influenza: A Seasonal and Pandemic Disease, Iran J. Med. Sci. 42 (2017) 2--13.

\bibitem{Barrientos2017} 
P.G. Barrientos, J.A. Rodr\'iguez, A. Ruiz-Herrera, Chaotic dynamics in the seasonally forced SIR epidemic model, J. Math. Biol. 75 (2017) 1655--1668.

\bibitem{Duarte2019} 
J. Duarte, C. Janu\'ario, N. Martins, S. Rogovchenko, Y. Rogovchenko, Chaos analysis and explicit series solutions to the seasonally forced SIR epidemic model, J. Math. Biol. 78 (2019) 2235--2258. 

\bibitem{Rashidinia2018}  
J. Rashidinia, M. Sajjadian, J. Duarte, C. Janu\'ario, N. Martins, On the Dynamical Complexity of a Seasonally Forced Discrete SIR Epidemic Model with a Constant Vaccination Strategy, Complexity 2018 (2018) 11 pages. 

\bibitem{Bilal2016} 
S. Bilal, B.K. Singh, A. Prasad, E. Michael, Effects of quasiperiodic forcing in epidemic models, Chaos 26 (2016) 8 pages.

\bibitem{Medio2009} 
A. Medio, M. Pireddu, F. Zanolin, Chaotic dynamics for maps in one and two dimensions: a geometrical method and applications to economics, Int. J. Bifurcat. Chaos 19 (2009) 3283--3309.

\bibitem{Li2011} 
J. Li, D. Blakeley, R.J. Smith, The failure of $\mathcal{R}_0$, Comput. Math. Methods Med. 2011 (2011) 17 pages. 

\bibitem{Jones2007} 
J.H. Jones, Notes on $\mathcal{R}_0$, California: Department of Anthropological Sciences 323 (2007) 19 pages.

\bibitem{LiTeng2017} 
J. Li, Z. Teng, G. Wang, L. Zhang, C. Hu, Stability and bifurcation analysis of an SIR epidemic model with logistic growth and saturated treatment, Chaos Solitons Fractals 99 (2017) 63--71.

\bibitem{ZhangChen1999} 
X.A. Zhang, L. Chen, The Periodic Solution of a Class of Epidemic Models, Comput. Math. with Appl. 38 (1999) 61--71.

\bibitem{Perez2019} 
A.G.C. P\'erez, E. Avila-Vales, G.E. Garc\'ia-Almeida, Bifurcation Analysis of an SIR model with Logistic Growth, Nonlinear Incidence, and Saturated Treatment, Complexity 2019 (2019) 21 pages.
 
\bibitem{ZhangLiu2008} 
X. Zhang, X. Liu, Backward bifurcation of an epidemic model with saturated treatment function, J. Math. Anal. 348 (2008) 433--443.

\bibitem{Rodrigues2020} 
A.A.P. Rodrigues, Unfolding a Bykov Attractor: From an Attracting Torus to Strange Attractors, J. Dyn. Diff. Equat. 2020 (2020) 35 pages.

\bibitem{MoraViana1993} 
L. Mora, M. Viana, Abundance of strange attractors, Acta Math. 171 (1993) 1--71. 

\bibitem{Yagasaki2002} 
K. Yagasaki, Melnikov's method and codimension-two bifurcations in forced oscillations, J. Differ. Equ. 185 (2002) 1--24.

\bibitem{GuckenheimerHolmes1983} 
J. Guckenheimer, P.J. Holmes, Nonlinear Oscillations, Dynamical Systems, and Bifurcations of Vector Fields. Applied Mathematical Sciences 42. Springer Verlag, New York, 1983.

\bibitem{WangYoung2003} 
Q. Wang, L.S. Young, Strange Attractors in Periodically-Kicked Limit Cycles and Hopf Bifurcations, Commun. Math. Phys. 240 (2003) 509--529. 

\bibitem{WangOtt2008} 
Q. Wang, W. Ott, Dissipative homoclinic loops of two-dimensional maps and strange attractors with one direction of instability, Commun. Pure Appl. Math. 64 (2011) 1439--1496. 

\bibitem{HPS} 
M.W. Hirsch, C.C. Pugh, M. Shub, Invariant Manifolds. Lecture Notes in Mathematics. Springer Verlag 583, 1977.

\bibitem{Passeggi2018} 
A. Passeggi, R. Potrie, M. Sambarino, Rotation intervals and entropy on attracting annular continua, Geom. Topol. 22 (2018) 2145--2186.

\bibitem{WangYoung2002} 
Q. Wang, L.S. Young, From Invariant Curves to Strange Attractors, Commun. Math. Phys. 225 (2002) 275--304.   

\bibitem{CastroRodrigues2021} 
L. Castro, A. Rodrigues, Torus-breakdown near a heteroclinic attractor: a case study, Int. J. Bifurcat. Chaos 31 (2021) 20 pages.

\bibitem{HerreraZanolin2014} 
A. Ruiz-Herrera, F. Zanolin, An example of chaotic dynamics in 3D systems via stretching along paths, Annali di Matematica 193 (2014) 163--185.

\bibitem{Duarte2021} 
J. Duarte, C. Janu\'ario, N. Martins, J. Seoane, M.A.F. Sanju\'an, Controlling infectious diseases: the decisive phase effect on a seasonal vaccination strategy, Int. J. Bifurcat. Chaos 31 (2021) 12 pages.





\end{thebibliography}
\end{document}